\documentclass{siamart220329}
\title{A Fourth-Order, Multigrid Cut-Cell Method For Solving Poisson's
  Equation in Three-Dimensional Irregular Domains}
\author{YIXIAO QIAN \thanks{School of Mathematical Sciences,
    Zhejiang University, 866 Yuhangtang Road, Haina Complex Building 2,
    HangZhou, Zhejiang, 310058 China.
    These authors equally contributed to the work and should be
    considered co-first authors.
  }
  \and WEIZHEN LI \footnotemark[1]
  \and YAN TAN \thanks{School of Mathematical Sciences,
    Zhejiang University, 866 Yuhangtang Road, Haina Complex Building 2,
    HangZhou, Zhejiang, 310058 China.}
\and QINGHAI ZHANG \thanks{(Corresponding author)
  School of Mathematical Sciences,
  Zhejiang University, 866 Yuhangtang Road, Haina Complex Building 2,
  HangZhou, Zhejiang, 310058 China (\url{qinghai@zju.edu.cn}).}
}
\usepackage{algorithm}
\usepackage{algorithmic}
\usepackage{amsfonts}
\usepackage{bbm}
\usepackage{bm}
\usepackage{enumitem}
\usepackage{tikz-cd}
\usepackage{tikz}
\usepackage[mode=buildnew]{standalone}
\usepackage{makecell}
\usetikzlibrary{backgrounds}
\usetikzlibrary{intersections}
\usetikzlibrary{math}
\usetikzlibrary{patterns}
\makeatletter
\newsavebox{\@brx}
\newcommand{\llangle}[1][]{\savebox{\@brx}{\(\m@th{#1\langle}\)}%
  \mathopen{\copy\@brx\mkern2mu\kern-0.9\wd\@brx\usebox{\@brx}}}
\newcommand{\rrangle}[1][]{\savebox{\@brx}{\(\m@th{#1\rangle}\)}%
  \mathclose{\copy\@brx\mkern2mu\kern-0.9\wd\@brx\usebox{\@brx}}}
\makeatother

\newcommand{\algorithmicprecondition}{\textbf{ Precondition:}} 
\newcommand{\PreCondition}{\item[\algorithmicprecondition]}
\newcommand{\algorithmicpostcondition}{\textbf{ Postcondition:}} 
\newcommand{\PostCondition}{\item[\algorithmicpostcondition]}
\newcommand{\algorithmicsideeffect}{\textbf{ Side effects:}} 
\newcommand{\SideEffect}{\item[\algorithmicsideeffect]}
\newcommand{\algorithmicbreak}{\textbf{break}}
\newcommand{\BREAK}{\STATE \algorithmicbreak}
\newcommand{\PNGDIR}{./png/}
\newcommand{\TIKZDIR}{./tikz/}
\usepackage{esint}
\usepackage{pgfplots}
\usepackage{scalefnt}
\usepackage{subcaption}

\definecolor{color1}{RGB}{145,30,180}
\definecolor{color2}{RGB}{245,130,48}
\definecolor{color3}{RGB}{230,25,75}

\begin{document}
\maketitle

\begin{abstract}
  We propose a fourth-order cut-cell method for solving
  Poisson's equations in three-dimensional irregular domains.
  Major distinguishing features of our method include
  (a) applicable to arbitrarily complex geometries,
  (b) high order discretization,
  (c) optimal complexity.
  Feature (a) is achieved by Yin space,
  which is a mathematical model for three-dimensional continua.
  Feature (b) is accomplished by poised lattice generation (PLG) algorithm,
  which finds stencils near the irregular boundary for polynomial fitting.
  Besides, for feature (c), we design a modified multigrid solver
  whose complexity is theoretically optimal by applying
  nested dissection (ND) ordering method.
\end{abstract}



\section{Introduction}
\label{sec:introduction}

In this article,
we consider the three-dimensional Poisson's equation
\begin{equation}
  \label{eq:poissonEquation}
  \Delta \varphi = f, \quad \text{in} ~ \Omega,
\end{equation}
where $\varphi: \mathbb{R}^3 \rightarrow \mathbb{R}$
is the unknown function,
and $\Omega$ is a bounded and connected domain in $\mathbb{R}^3$.
Poisson's equation, which is a fundamental
elliptic partial differential equation,
has broad applications in numerous scientific and engineering problems,
such as electrostatics, fluid dynamics, and thermal analysis.
For instance, in the field of fluid mechanics,
solving the incompressible Navier-Stokes equations (INSE)
via projection methods
\cite{brown2001accurate,johnston2004accurate,
  liu2007stability,zhang2014fourth,zhang2016gepup}
involves solving multiple Poisson's equations
with different boundary conditions.
Accurately and efficiently solving these Poisson's equations
in three-dimensional irregular domains
is vital for advancing simulations and analysis in these areas.

Numerous classical numerical methods
have been developed for solving~(\ref{eq:poissonEquation})
in rectangular domains, whether two-dimensional or three-dimensional.
However, most real-world problems are highly complex,
making it challenging to directly apply these conventional methods.
There is an urgent need for
developing advanced numerical techniques
capable of handling the complex computational domain boundaries.

One popular approach is the finite element method (FEM),
which is known for its high adaptability, flexibility and accuracy.
FEM employs unstructured grids to partition the domain
into subregions, such as triangles,
offering the ability to accurately represent
complex geometries and boundary conditions.
However,
these unstructured grids demand the storage of more information
compared to structured grids,
resulting in increased memory overhead.
Furthermore, the non-continuous nature of information storage
diminishes the efficiency of memory access.
FEM is also highly mesh-dependent \cite{brenner2008mathematical},
but generating high-order conforming mesh representations
for complex three-dimensional domains is both challenging and costly.
Another widely favored approach for handling complex geometries
is the immersed boundary method (IBM)
\cite{peskin1972flow,peskin2002immersed,
  tseng2003ghost,verzicco2023immersed}
based on finite-difference schemes.
This method embeds the irregular boundary into a
Cartesian structured grid
without performing Boolean operations.
Boundary conditions are enforced
by adding a volumetric forcing term into the governing equations,
either explicitly or implicitly.
Although IBM offers flexibility and simplicity
in managing complex geometries,
maintaining accuracy and stability near arbitrarily complex boundaries,
particularly in high Reynolds number flows,
remains challenging.
Additionally, IBM is strongly problem-dependent and
typically associated with low-order accuracy.

The cut-cell method, also known as
the Cartesian grid method or embedded boundary (EB) method,
provides an alternative by embedding irregular domains
within a regular Cartesian grid
and generating cut cells through the intersection of
cell boundaries with the geometric boundary.
EB method retains the simplicity of Cartesian grid
while adapting to complex geometries.
It can take advantage of many well-established techniques
from finite difference or finite volume methods,
such as high-order conservative schemes
for incompressible flows \cite{morinishi1998fully},
the multigrid algorithm \cite{briggs2000multigrid}
for elliptic equations,
and AMR algorithms \cite{dezeeuw1993adaptively, pember1995adaptive}.
But meanwhile, for high-order discretization,
several related issues still require effective solutions.
For instance, the cut-cell method often encounters challenges such as
degraded accuracy at the embedded boundaries
and instability caused by the small cut-cell problem
\cite{berger2021state, giuliani2022weighted}.
Furthermore, achieving optimal-complexity solvers
for the corresponding discrete linear systems
remains an active area of research.

Second-order cut-cell methods have been successfully employed to solve
Poisson's equations
\cite{gibou2005fourth,johansen1998cartesian,schwartz2006cartesian},
heat equations
\cite{mccorquodale2001cartesian,schwartz2006cartesian}
and Navier-Stokes equations
\cite{kirkpatrick2003representation,trebotich2015adaptive}.
Recently, Devendran et al. developed
a fourth-order EB method for Poisson's equations
\cite{devendran2017fourth},
and Overton-Katz et al. introduced a fourth-order EB method
for unsteady Stokes equations
\cite{overton2023fourth}.
They utilize weighted least squares to derive formulas for
high-order discretizations.
However, these methods do not provide a general framework
for generating stencils
and lack the flexibility to be easily extended to
arbitrarily complex geometries.
Additionally,
most existing approaches depend on the multigrid solver implemented
by EBChombo \cite{colella2014ebchombo}.
And there is an absence of comprehensive complexity analysis for their multigrid solvers.

Notably, our research group
has proposed a novel fourth-order cut-cell method
\cite{li2024fourthorder}
designed for two-dimensional Poisson's equations.
This method showcases
the ability to handle arbitrarily complex domains
while employing a multigrid solver with optimal complexity.
In this study,
we build upon this method,
extending it to three-dimensional Poisson's equations
while preserving its core strengths.

The above discussion motivates questions as follows:

\begin{enumerate}[label=(Q-\arabic*)]
\item Given arbitrarily complex computational domains,
  is there an accurate and efficient representation of such domains?
  \label{Q:boundaryRepresentation}
\item Cut cells with a small volume fraction
  may induce stability issues.
  Is it possible to devise
  an effective merging algorithm
  to address this challenge?
  \label{Q:processMerging}
\item Conventionally, achieving a high-order discretization
  of differential operators requires specialized techniques
  and complex computations.
  Is it feasible to design a high-order discretization method
  with low computational cost
  that can be applied to arbitrarily complex domains?
  \label{Q:processDiscretization}
\item Is there a viable strategy to
  solve the discrete linear system efficiently and
  with theoretically optimal complexity?
  \label{Q:solveResultingSystem}
\end{enumerate}
In this paper, we provide positive answers to
all the above challenges by presenting a
fourth-order cut-cell method for solving Poisson's equations
in three-dimensional irregular domains,
with extensibility to constant-coefficient elliptic equations.

For \ref{Q:boundaryRepresentation},
in the two-dimensional case,
Li, Zhu and Zhang \cite{li2024fourthorder}
make use of the theory of two-dimensional Yin space
\cite{zhang2020boolean},
in which each Yin set has a simple and accurate representation
that facilitates geometric and topological queries
via polynomial spline curves.
Similarly, in the three-dimensional case,
we employ the three-dimensional Yin space theory
\cite{zhang2024boolean}.
In specific, when dealing with the irregular boundaries
of the computational domain,
we utilize the least squares method
to fit piecewise quadratic polynomial surfaces for their approximation.
Then the Boolean intersection operation
of Yin space is applied to determine
the accurate representation of each cut cell.

For \ref{Q:processMerging},
we develop a systematic algorithm for merging the small cells
that have a volume fraction below a user-specified threshold.
Specifically, we pay special attention to the case of
\emph{multi-component} cells,
where a single cell comprises multiple connected components.

For \ref{Q:processDiscretization},
the discretization method from \cite{li2024fourthorder}
based on the poised lattice generation (PLG) algorithm \cite{PLG}
is implemented.
The PLG algorithm
generates stencils to fit complete multivariate polynomials
via weighted least squares method,
enabling high-order discretization of linear differential operators.
This method is applicable to various boundary conditions
and nonlinear differential operators.

For \ref{Q:solveResultingSystem},
we modify the multigrid components as described in
\cite{li2024fourthorder}
to adapt to irregular domains
by coupling the smoothing operator with LU factorization.
The optimal complexity of the modified multigrid algorithm
is theoretically demonstrated,
which,
while trivial in two-dimensional case,
presents challenges in three dimensions.
To achieve optimal complexity, the nested dissection ordering method
\cite{george1973nested,karypis1998fast,lipton1979generalized}
is applied to renumber the cells near embedded boundaries,
thereby efficiently reducing the complexity of the LU factorization for
the matrix block corresponding to these cells.

Despite significant advancements in solving the Poisson's equations
within three-dimensional irregular domains, existing methodologies often
fall short in addressing all four critical challenges
identified in this research.
To the best of our knowledge, no single method in the literature
has successfully and simultaneously tackled all four challenges
in a comprehensive and efficient manner.
By systematically addressing each of these problems,
the novel approach proposed in this study represents
a meaningful advancement in the field,
offering a promising framework that can pave the way
for more accurate and robust solutions to Poisson's equations
in three-dimensional complex geometries,
with broad applicability across diverse fields.


\section{Roadmap}
\label{sec:roadmap}

In this section,
we provide an overview of our method,
leaving additional details in subsequent sections.

\subsection{Yin Space}
\label{sec:yin-space}

To establish a solid foundation for
describing continua's complex topology,
large geometric deformations, and topological
changes such as merging in the context of multiphase flow,
\emph{Yin space},
a mathematical modeling space,
was proposed for continua with
two-dimensional \cite{zhang2020boolean}
and three-dimensional \cite{zhang2024boolean}
arbitrarily complex topology.

\begin{definition}[Yin space \cite{zhang2024boolean}]
  A \emph{Yin set} $\mathcal{Y}$ in $\mathbb{R}^3$ is a
  regular open semianalytic set
  whose boundary is bounded.
  The class of all such Yin sets constitutes the \emph{Yin space} $\mathbb{Y}$.
\end{definition}

\begin{theorem}[Zhang and Li \cite{zhang2020boolean}]
  \label{thm:yinSetsFormABooleanAlgebra}
   The algebra
   $\mathbf{Y}:=({\mathbb Y},\ \cup^{\perp\perp},\ \cap,\ \, ^{\perp},\
   \emptyset,\ \mathbb{R}^3)$
   is a Boolean algebra.
\end{theorem}

\begin{definition}
  \label{def:gluedSurface}
  A \emph{glued surface} is a compact 2-manifold or its quotient space,
  whose
  quotient map glues the compact manifold along the subsets homeomorphic to
  a one-dimensional CW complex, and its complement has exactly two connected
  components.
\end{definition}

\begin{theorem}
  \label{thm:yinsetRep}
  For a Yin set $\mathcal{Y} \neq \emptyset, \mathbb{R}^{3}$, its
  boundary can be uniquely decomposed into several glued surfaces, which
  can be further oriented such that
  \begin{equation*}
    \mathcal{Y} = \bigcup\nolimits^{\perp\perp}_{j}
    \mathop{\bigcap}\limits_{i}
    \mathrm{int}(\mathcal{S}_{j,i}),
  \end{equation*}
where $j$ is the index of connected components of $\mathcal{Y}$ and
$\mathcal{S}_{j,i}$'s are oriented glued surfaces without pairwise proper
intersections.
\end{theorem}

In \cite{zhang2024boolean},
all surface patches forming glued surfaces are triangular.
To achieve higher accuracy and smoothness,
these triangular patches can be replaced with
polynomial surfaces, Bézier surfaces or B-spline surfaces.
In this paper,
we employ polynomial surfaces generated through least squares fitting
to construct the Yin sets,
as detailed in Section \ref{sec:geometric-characterization}.

\subsection{Grid Construction}

Let $\Omega\in\mathbb{Y}$ denote
the three-dimensional computational domain,
and $R$ be a rectangular region enclosing $\Omega$,
which is
uniformly partitioned into a collection
of rectangular cells defined by
\begin{equation*}
  C_{\mathbf{i}} = \Big(\mathbf{x}_O + \mathbf{i}h,
  \mathbf{x}_O + (\mathbf{i} + \mathbbm{1})h\Big),
\end{equation*}
where $\mathbf{x}_O$ is a fixed origin in the coordinate system,
$h$ represents the uniform spatial step size,
$\mathbf{i} \in \mathbb{Z}^3$
is a multi-index and $\mathbbm{1}\in\mathbb{Z}^3$
is the multi-index with
all components equal to one.
The upper and lower faces of the cell $C_\mathbf{i}$
along the $d$-th dimension are respectively denoted by
\begin{align*}
  F_{\mathbf{i} + \frac{1}{2}\mathbf{e}^d}
  &= \Big(\mathbf{x}_O + (\mathbf{i} + \mathbf{e}^d)h,
    \mathbf{x}_O + (\mathbf{i} + \mathbbm{1})h\Big),\\
  F_{\mathbf{i} - \frac{1}{2}\mathbf{e}^d}
  &= \Big(\mathbf{x}_O, \mathbf{x}_O +
    (\mathbf{i} + \mathbbm{1} - \mathbf{e}^d)h\Big),
\end{align*}
where $\mathbf{e}^d \in \mathbb{Z}^D$ is a multi-index with $1$
as its $d$-th component and $0$ otherwise.

Embedding $\Omega$ into the Cartesian grid $R$,
we define the cut cells by
\begin{equation*}
  \mathcal{C}_{\mathbf{i}} := C_{\mathbf{i}} \cap \Omega,
\end{equation*}
the cut faces by
\begin{equation*}
  \mathcal{F}_{\mathbf{i} + \frac{1}{2}\mathbf{e}^d}
  := F_{\mathbf{i} + \frac{1}{2}\mathbf{e}^d} \cap \Omega,
  \mathcal{F}_{\mathbf{i} - \frac{1}{2}\mathbf{e}^d}
  := F_{\mathbf{i} - \frac{1}{2}\mathbf{e}^d} \cap \Omega,
\end{equation*}
and the irregular boundary surfaces
(i.e., the portion of domain boundary contained in cut cells) by
\begin{equation*}
  \mathcal{S}_{\mathbf{i}} := C_{\mathbf{i}} \cap \partial \Omega.
\end{equation*}
Let $\|\mathcal{C}_{\mathbf{i}}\|$ denote
the volume of $\mathcal{C}_{\mathbf{i}}$,
and
$\|\mathcal{F}_{\mathbf{i} + \frac{1}{2}e^{\mathbf{d}}}\|,
\|\mathcal{S}_{\mathbf{i}}\|$
denote the area of
$\mathcal{F}_{\mathbf{i}+\frac{1}{2}e^{\mathbf{d}}},
\mathcal{S}_{\mathbf{i}}$ respectively.
Particularly,
$\mathcal{C}_{\mathbf{i}}$ is said to be an \emph{interior cell}
if $\mathcal{C}_{\mathbf{i}} = C_{\mathbf{i}}$,
an \emph{exterior cell} if $\mathcal{C}_{\mathbf{i}} = \emptyset$,
and a \emph{cut cell} otherwise.

\subsection{Spatial Discretization}

Consider the discretization of the equation (\ref{eq:poissonEquation}) 
with boundary condition
\begin{equation}
  \label{eq:poissonBoundaryCondition}
  \mathcal{N} \varphi = g, \quad \text{on} ~ \partial \Omega,
\end{equation}
where $\mathcal{N}$ represents the boundary condition operator.
For instance, $\mathcal{N} = \mathcal{I}$ for Dirichlet conditions,
$\mathcal{N} = \frac{\partial}{\partial \mathbf{n}}$ for Neumann conditions,
and $\mathcal{N} = \gamma_1 + \gamma_2 \cdot \frac{\partial}{\partial \mathbf{n}}
(\gamma_1, \gamma_2 \in \mathbb{R})$ for Robin conditions.

Denote the cell-averaged value of
a scalar function $\varphi$ over cell $\mathcal{C}_{\mathbf{i}}$ by
\begin{equation*}
  \langle \varphi \rangle_{\mathbf{i}} =
  \frac{1}{\|\mathcal{C}_{\mathbf{i}}\|} \int_{\mathcal{C}_{\mathbf{i}}}
  \varphi(\mathbf{x}) \mathrm{d} \mathbf{x},
\end{equation*}
the face-averaged value of $\varphi$ over the
face $\mathcal{F}_{\mathbf{i} + \frac{1}{2}\mathbf{e}^d}$ by
\begin{equation*}
  \langle \varphi \rangle_{\mathbf{i} + \frac{1}{2}\mathbf{e}^d}
  = \frac{1}{\|\mathcal{F}_{\mathbf{i}
      + \frac{1}{2}\mathbf{e}^d}\|}\int_{\mathcal{F}_{\mathbf{i}
      + \frac{1}{2}\mathbf{e}^d}}
  \varphi(\mathbf{x})\mathrm{d} \mathbf{x},
\end{equation*}
and the face-averaged value of $\varphi$ over the
irregular boundary surface $\mathcal{S}_{\mathbf{i}}$ by
\begin{equation*}
  \llangle \varphi \rrangle_{\mathbf{i}} =
  \frac{1}{\|\mathcal{S}_{\mathbf{i}}\|} \int_{\mathcal{S}_{\mathbf{i}}}
  \varphi(\mathbf{x}) \mathrm{d} \mathbf{x}.
\end{equation*}

For a cell $\mathcal{C}_{\mathbf{i}}$,
if none of the cells within the set
$\{\mathcal{C}_{\mathbf{k}}: \mathbf{k} = \mathbf{i},
\mathbf{i} \pm \mathbf{e}^d,
\mathbf{i} \pm 2\mathbf{e}^d, d = 0,1,2\}$
contain any irregular boundary surfaces
(i.e., they are all interior cells),
then standard formulas can be applied to
derive the discrete Laplacian operator
\begin{equation}
  \label{eq:regularLaplaceDiscretization}
  \langle\Delta \varphi\rangle_{\mathbf{i}}=
  \frac{1}{12 h^2} \sum_d
  \Big( -\langle\varphi\rangle_{\mathbf{i}+2 \mathbf{e}^d}
    +16\langle\varphi\rangle_{\mathbf{i}+\mathbf{e}^d}
    -30\langle\varphi\rangle_{\mathbf{i}}
    +16\langle\varphi\rangle_{\mathbf{i}-\mathbf{e}^d}
    -\langle\varphi\rangle_{\mathbf{i}
    -2 \mathbf{e}^d} \Big) +\mathrm{O}\left(h^4\right).
\end{equation}
For cells near the regular boundaries,
ghost cells (see \cite{zhang2012fourth})
are filled based on specific boundary condition
to facilitate above standard discretization schemes.
Particularly, for a Dirichlet boundary condition where
$\langle \varphi \rangle_{\mathbf{i} + \frac{1}{2}\mathbf{e}^d}
= \langle g\rangle_{\mathbf{i} + \frac{1}{2}\mathbf{e}^d}$,
the ghost cell values are filled with
\begin{align*}
  \langle\varphi\rangle_{\mathbf{i}+\mathbf{e}^{d}}
  &=\frac{1}{12}\left(3\left\langle\varphi\right\rangle_{\mathbf{i}-3\mathbf{e}^{d}}
    -17\left\langle\varphi\right\rangle_{\mathbf{i}-2\mathbf{e}^{d}}
    +43\left\langle\varphi\right\rangle_{\mathbf{i}-\mathbf{e}^{d}}
    -77\left\langle\varphi\right\rangle_{\mathbf{i}}
    +60\left\langle\varphi\right\rangle_{\mathbf{i}
      +\frac{1}{2}\mathbf{e}^{d}}\right)+\mathrm{O}(h^{5}),\\
  \left\langle\varphi\right\rangle_{\mathbf{i}+2\mathbf{e}^{d}}
  &=\frac{1}{12}\left(27\left\langle\varphi\right\rangle_{\mathbf{i}-3\mathbf{e}^{d}}
    -145\left\langle\varphi\right\rangle_{\mathbf{i}-2\mathbf{e}^{d}}
    +335\left\langle\varphi\right\rangle_{\mathbf{i}-\mathbf{e}^{d}}
    -505\left\langle\varphi\right\rangle_{\mathbf{i}}
    +75\left\langle\varphi\right\rangle_{\mathbf{i}+\frac{1}{2}\mathbf{e}^{d}}\right)
  +O(h^{5}).
\end{align*}
Similarly, for a Neumann boundary condition with
$\langle \frac{\partial \varphi}{\partial x_d} \rangle_{\mathbf{i} + \frac{1}{2}\mathbf{e}^d}
= \langle g\rangle _{\mathbf{i} + \frac{1}{2}\mathbf{e}^d}$,
fourth-order interpolation yields
\begin{align*}
  \langle\varphi\rangle_{\mathbf{i}+\mathbf{e}^{d}}
  &=\frac{1}{10}\left(\langle\varphi\rangle_{\mathbf{i}-3\mathbf{e}^{d}}
    -5\left\langle\varphi\right\rangle_{\mathbf{i}-2\mathbf{e}^{d}}
    +9\left\langle\varphi\right\rangle_{\mathbf{i}-\mathbf{e}^{d}}
    +5\left\langle\varphi\right\rangle_{\mathbf{i}}
    +12h\left\langle\frac{\partial\varphi}{\partial\mathbf{n}}\right\rangle
    _{\mathbf{i}+\frac{1}{2}\mathbf{e}^{d}}\right)+\mathrm{O}(h^{5}),\\
  \langle\varphi\rangle_{\mathbf{i}+2\mathbf{e}^{d}}
  &=\frac{1}{2}\left(3\left\langle\varphi\right\rangle_{\mathbf{i}-3\mathbf{e}^{d}}
    -15\left\langle\varphi\right\rangle_{\mathbf{i}-2\mathbf{e}^{d}}
    +29\left\langle\varphi\right\rangle_{\mathbf{i}-\mathbf{e}^{d}}
    -15\left\langle\varphi\right\rangle_{\mathbf{i}}
    +12h\left\langle\frac{\partial\varphi}{\partial\mathbf{n}}\right\rangle
    _{\mathbf{i}+\frac{1}{2}\mathbf{e}^{d}}\right)+\mathrm{O}(h^{5}).
\end{align*}

Although the standard discrete Laplacian operator for inner cells
can be derived straightforwardly from
(\ref{eq:regularLaplaceDiscretization}),
obtaining high-order discretization for cells
around the irregular boundaries is
significantly more challenging
due to the complexity of the boundaries.
Following the approach presented in
Section 3 of \cite{li2024fourthorder},
we undertake the following steps to
derive the high-order discretization.

Firstly, the finite-volume poised lattice generation
(FV-PLG, see Section \ref{sec:poised-lattice-generation})
technique is employed to establish
a stencil for polynomial interpolation.
Given a cell $\mathcal{C}_{\mathbf{i}}$ around the irregular boundaries,
FV-PLG method generates a collection of sites $\mathcal{X}(\mathbf{i})$
near $\mathcal{C}_{\mathbf{i}}$ for polynomial fitting in $\Pi_n^D$.
This set $\mathcal{X}(\mathbf{i})$ can be expressed as
\begin{equation*}
  \mathcal{X}(\mathbf{i}) =
  \left\{ \mathcal{C}_{\mathbf{j}_1}, \cdots,
    \mathcal{C}_{\mathbf{j}_N} \right\}
  \cup
  \left\{ \mathcal{S}_{\mathbf{j}_{N+1}},\cdots,
    \mathcal{S}_{\mathbf{j}_{N + N^{\prime}}} \right\}.
\end{equation*}

Secondly, the identified stencil $\mathcal{X}(\mathbf{i})$
is used to perform a local $D$-variable polynomial fitting.
Specifically,
a complete $n$-degree polynomial with $D$-variable is constructed as
\begin{equation*}
  p(\mathbf{x}) = \sum\limits_{j = 1}^N
  \alpha_j \phi_j(\mathbf{x}) \in \Pi_n^D,
\end{equation*}
where $\Pi_n^D$ is the vector space of all D-variate polynomials
of degree no more than $n$ with real coefficients,
$\{\phi_j\}_{j = 1}^N$ constitutes a basis of $\Pi_n^D$,
and the coefficient vector
$\bm{\alpha} = [\alpha_1,\cdots,\alpha_n]^T$
is the solution of the weighted least squares problem
\begin{equation*}
  \min \limits_{\bm{\alpha}} \sum\limits_{k = 1}^N \omega_k
  \left| \langle p \rangle_{\mathbf{j}_k}
    - \langle\varphi\rangle_{\mathbf{j}_k} \right|^2
  + \sum\limits_{k = N+1}^{N + N^{\prime}} \omega_k
  \left| \llangle \mathcal{N} p \rrangle _{\mathbf{j}_k}
    - \llangle \mathcal{N} \varphi\rrangle_{\mathbf{j}_k} \right|^2,
\end{equation*}
where $\omega_k$ depends on the relative position
between $\mathcal{C}_{\mathbf{j}_k}$ and $\mathcal{C}_{\mathbf{i}}$.

Finally, applying the Laplacian operator over $p(\mathbf{x})$
yields the approximation, i.e.,
\begin{equation}
  \label{eq:irregularLaplaceDiscretization}
  \langle \mathcal{L} \varphi \rangle_{\mathbf{i}}
  = \langle \mathcal{L} p \rangle_{\mathbf{i}} + O(h^{n-1}).
\end{equation}
In this paper, we fit polynomials of degree 4, which yield $O(h^3)$
truncation error and $O(h^4)$ solution error.

It is worth noting that small cells
significantly impact the robustness of
the approximation and the linear solver.
To address this issue,
we employ a merging algorithm to merge small cells into larger ones,
as demonstrated in Section \ref{sec:merging-algorithm}.

\subsection{Discrete Poisson's Equation}

By coupling the fourth-order difference formula
(\ref{eq:regularLaplaceDiscretization})
with the FV-PLG approximation (\ref{eq:irregularLaplaceDiscretization}),
we ultimately derive the discretization
of (\ref{eq:poissonEquation}) with boundary condition
(\ref{eq:poissonBoundaryCondition}) as
\begin{equation*}
  L\hat{\varphi} + N\hat{g} = \hat{f},
\end{equation*}
where $\hat{\varphi}, \hat{f}$ denote the vectors of
cell-averaged values of the function $\varphi$ and $f$ respectively,
$\hat{g}$ represents the vector of boundary face-averaged values
corresponding to the boundary condition $g$,
and $L, N$ are both matrix operators.
It can be transformed into a residual form as
\begin{equation*}
  L\hat{\varphi} = \hat{r} := \hat{f} - N\hat{g},
\end{equation*}
which can be further partitioned into two row blocks:
\begin{equation}
  \label{eq:multigridLinearSystem}
  \left[
    \begin{array}{cc}
      L_{11}&L_{12}\\
      L_{21}&L_{22}
    \end{array}
  \right] \left[
    \begin{array}{c}
      \hat{\varphi}_1\\
      \hat{\varphi}_2
    \end{array}
  \right] = \left[
    \begin{array}{c}
      \hat{r}_1\\
      \hat{r}_2
    \end{array}
  \right],
\end{equation}
where the splitting
$\hat{\varphi} = [\hat{\varphi}_1, \hat{\varphi}_2]^T$
is based on the type of discretization.
If the regular difference formula
(\ref{eq:regularLaplaceDiscretization})
is applied to $\mathcal{C}_{\mathbf{i}}$,
then the cell-average $\langle \varphi \rangle_{\mathbf{i}}$
is contained in $\hat{\varphi}_1$;
otherwise, it is included in $\hat{\varphi}_2$.
As a result,
$L_{11}$ exhibits a regular structure similar to
that obtained by directly applying standard
discretizations of Poisson's equations in regular domains,
and other matrix blocks $L_{12},L_{21},L_{22}$ has no more explicit
structures beyond sparsity.

In this paper, we employ the multigrid method
to solve the linear system (\ref{eq:multigridLinearSystem}).
However, it is notable that the FV-PLG discretization
prohibits the direct application of
traditional geometric multigrid methods.
On the one hand, the Gauss-Seidel or (weighted) Jacobi
iterations do not guarantee convergence
due to the indefinite and asymmetrical structures
of $L_{12},L_{21},L_{22}$ in
(\ref{eq:multigridLinearSystem}).
On the other hand,
simple grid-transfer operators cannot directly be
applied near the irregular boundary,
as the cells' volumes are non-uniform.
We introduce a modified version of the geometric multigrid method
to address these limitations.
Additionally, we demonstrate that our modified multigrid method
achieves optimal complexity,
as detailed in Section \ref{sec:multigrid}.


\section{Geometric Characterization}
\label{sec:geometric-characterization}

\subsection{Boundary Fitting}

We adopt piecewise quadratic polynomial surfaces to approximate
the boundary $\partial\Omega$ of the computational domain.
Inside every cut cell $\mathcal{C}_{\mathbf{i}}$,
a selection of points is made from $\partial \Omega$,
and a quadratic polynomial surface $w = p(u,v)$ is fitted
by solving a least squares problem,
where $u,v,w$ represent a permutation of the three axes $x,y,z$.
The region enclosed by these approximating surfaces
is denoted as $\Omega^{\prime}$,
which is the approximation of $\Omega$ in $\mathbb{Y}$.

\begin{theorem}
  \label{thm:leastSquaresFittingError1D}
  Consider a function $f\in\mathcal{C}^3\big([a_0,b_0]\big)$.
  If $N (N \geq 3)$ points $\{x_i\}_{i = 1}^N$
  are distributed in $[a_0, b_0]$
  and employed in a least squares fit
  for the quadratic polynomial $p(x) = ax^2 + bx + c$,
  the resulting approximation satisfies
  \begin{equation*}
    f(x) = p(x) + O(h^3),\ \forall x\in[a_0,b_0],
  \end{equation*}
  where $h = b_0 - a_0$.
\end{theorem}

\begin{proof}
  Without loss of generality, we consider the interval to be $[0, h]$.
  The least squares solution $[a, b, c]^T$
  satisfies the normal equations:
    \begin{equation*}
        A^TA \begin{bmatrix}
            a\\
            b\\
            c
          \end{bmatrix} = A^T F,\  \text{where} \ A =
          \begin{bmatrix}
            x_1^2 & x_1 & 1\\
            \vdots & \vdots & \vdots\\
            x_N^2 & x_N & 1
          \end{bmatrix}
          ,\ F=\begin{bmatrix}
            f(x_1)\\
            \vdots\\
            f(x_N)
           \end{bmatrix}.
    \end{equation*}
    According to matrix multiplication and the Cramer's rule, we have
    \begin{align}
      \label{eq:cramerA}
        &\quad\quad\quad A^TA = \begin{bmatrix}
            \sum x_i^4 & \sum x_i^3 & \sum x_i^2\\
            \sum x_i^3 & \sum x_i^2 & \sum x_i\\
            \sum x_i^2 & \sum x_i & \sum 1
           \end{bmatrix},\quad\quad\quad
        a = \frac{1}{\mathrm{det}(A^TA)} \mathrm{det}\begin{bmatrix}
            \sum x_i^2f(x_i) & \sum x_i^3 & \sum x_i^2\\
            \sum x_i f(x_i) & \sum x_i^2 & \sum x_i\\
            \sum f(x_i) & \sum x_i & \sum 1
           \end{bmatrix},\\
      \label{eq:cramerB}
        b &= \frac{1}{\mathrm{det}(A^TA)} \mathrm{det}\begin{bmatrix}
            \sum x_i^4 & \sum x_i^2f(x_i) & \sum x_i^2\\
            \sum x_i^3 & \sum x_i f(x_i) & \sum x_i\\
            \sum x_i^2 & \sum f(x_i) & \sum 1
           \end{bmatrix},
        c = \frac{1}{\mathrm{det}(A^TA)} \mathrm{det}\begin{bmatrix}
            \sum x_i^4 & \sum x_i^3 & \sum x_i^2f(x_i)\\
            \sum x_i^3 & \sum x_i^2 & \sum x_if(x_i)\\
            \sum x_i^2 & \sum x_i & \sum f(x_i)
           \end{bmatrix}.
    \end{align}
    Then we get the estimation $\mathrm{det}(A^TA) = O(h^6)$.
    By Taylor's theorem, we have
    \begin{align}
      \label{eq:taylorExpansionOff(x)}
        f(x) &= f(0) + f'(0)x + \frac{1}{2}f''(0)x^2 + O(h^3),\ \forall x\in[0, h],\\
        \label{eq:taylorExpansionOff(x_i)}
        f(x_i) &= f(0) + f'(0)x_i + \frac{1}{2}f''(0)x_i^2 + O(h^3),\ \forall i.
    \end{align}
    Substituting $(\ref{eq:taylorExpansionOff(x)})$ and
    $(\ref{eq:taylorExpansionOff(x_i)})$
    into equations $(\ref{eq:cramerA})$ and $(\ref{eq:cramerB})$,
    we arrive at
    \begin{equation}
      \label{eq:fittingCoef}
      a = \frac{1}{2}f^{\prime\prime}(0) + O(h),
      \quad b = f^{\prime}(0) + O(h^2),
      \quad c = f(0) + O(h^3).
    \end{equation}
    Therefore, we have
    \begin{align*}
      ax^2 + bx + c - f(x) &= \Big( \frac{1}{2}f^{\prime\prime}(0) + O(h)\Big)x^2
                             + \Big(f^{\prime}(0) + O(h^2)\Big) x
                             + f(0) + O(h^3)\\
                           &- \Big( f(0) + f^{\prime}(0)x
                             + \frac{1}{2}f^{\prime\prime}(0)x^2
                             + O(h^3) \Big)\\
                           &=O(h^3), \forall x\in[0, h].
    \end{align*}
\end{proof}

Using the same logical reasoning applied in
Theorem \ref{thm:leastSquaresFittingError1D},
we can derive an analogous conclusion for the two-dimensional case.

\begin{corollary}\label{coro:leastSquaresFittingError2D}
  Let $f\in\mathcal{C}^3\big([a_0,a_0+h]\times[b_0,b_0+h]\big)$,
  where $h \in \mathbb{R}^+$.
  By selecting $N$ points $\{(x_i,y_i)\}_{i = 1}^N$ within the rectangle
  $[a_0, a_0+h]\times[b_0,b_0+h]$
  and employing the $\{(x_i,y_i,f(x_i,y_i))\}_{i = 1}^N$
  data set for least squares fitting of
  a quadratic polynomial $p(x,y) = ax^2+bxy+cy^2+dx+ey+g$,
  we have
  \begin{equation*}
    f(x,y) = p(x,y) + O(h^3),
    \forall (x,y) \in[a_0,a_0+h]\times[b_0,b_0+h].
  \end{equation*}
\end{corollary}

For any cut cell,
let $V_f$ denote the intersection region yielded by
the exact surface, whereas $V_{p}$ denotes
the corresponding region yielded by
the approximate least squares surface.
Furthermore, let $S_{f}$ and $S_{p}$ represent
the irregular boundary surfaces within $V_{f}$ and $V_{p}$, respectively.
For this particular boundary approximation,
we present evaluations of the area and
surface integral errors over $S_{f}$ and $S_{p}$,
as well as the volume and volume integral errors
within $V_{f}$ and $V_{p}$.

\begin{theorem}\label{thm:leastSquaresSurfaceAreaError}
  Consider a cut cell in the domain
  $\Omega_0 = [x_0, x_0+h]\times[y_0, y_0+h]\times[z_0,z_0+h]$.
  Let height function $f(x,y)$ represent
  the exact surface within this cell,
  and $p(x,y)$ denote its least squares approximation.
  The error in the surface area satisfies
  \begin{equation}
    \label{eq:leastSquaresSurfaceAreaError}
    \|S_f\| = \|S_p\| + O(h^4).
  \end{equation}
\end{theorem}

\begin{proof}
  Let $D_f$ and $D_p$ denote the projection areas of $S_f$ and $S_p$
  onto the region $[x_0, x_0+h]\times[y_0, y_0+h]$, respectively.
  We have
  \begin{align*}
    &\left| ~ \left\|S_f\right\| - \left\|S_p\right\| ~ \right| \\
    =&\left|\int_{D_f} \sqrt{1+f_x^2+f_y^2}\mathrm{d}x\mathrm{d}y
      - \int_{D_p} \sqrt{1+p_x^2+p_y^2} \mathrm{d}x\mathrm{d}y\right|\\
    \le&\left| \int_{D_f \cap D_p}
        \frac{f_x^2+f_y^2 - p_x^2 -p_y^2}{\sqrt{1+f_x^2+f_y^2}
        + \sqrt{1+p_x^2+p_y^2}} \mathrm{d}x\mathrm{d}y\right|
        + \left|\int_{D_f \oplus D_p} O(1)\mathrm{d}x\mathrm{d}y \right|\\
    = &err_1 + err_2,
  \end{align*}
  According to Corollary \ref{coro:leastSquaresFittingError2D},
  we have $f_x(x,y) = p_x(x,y) + O(h^2)$ and
  $f_y(x,y) = p_y(x,y) + O(h^2)$.
  Hence, we obtain
  \begin{equation}
  \label{eq:err1Estimation}
    err_1 \leq O(h^2)\cdot \|S_{D_f \cap D_p}\| = O(h^4).
  \end{equation}
  For $D_f\oplus D_p$, consider the area enclosed
  by the intersection lines of the two surfaces with the planes
  $z = z_0$ and $z = z_0 + h$.
  Without loss of generality,
  we consider the scenario depicted in Figure
  \ref{fig:figureForThmLeastSquaresSurfaceError}.
  Let the local expressions of
  the intersection lines with respect to $x$ and $y$ be denoted as
  $\phi_f^x(x)$, $\phi_f^y(y)$, $\phi_p^x(x)$, and $\phi_p^y(y)$.
  We can estimate the area as follows:
  \begin{equation*}
    \|S_{D_f\oplus D_p}\|
    \le \int_{y_0}^{y^*} |\phi_f^y - \phi_p^y| \mathrm{d}y
    + \int_{x^*}^{x_0+h} |\phi_f^x - \phi_p^x|\mathrm{d}x.
  \end{equation*}
  For any $y\in(y_0,y^{*})$, since points $(\phi_p^y(y),y,z_0+h)$
  and $(\phi_f^y(y),y,z_0+h)$ lie on the intersection lines,
  we have
  \begin{equation}
    \label{eq:estimateErrorOfLines}
    z_0+h = p\left(\phi_p^y(y),y\right)
    = f \left(\phi_f^y(y),y\right)
    = p \left(\phi_f^y(y),y\right) + O(h^3),
  \end{equation}
  where the last step follows from
  Corollary \ref{coro:leastSquaresFittingError2D}.
  Using the Taylor expansion of $p(\phi_p^y(y),y)$,
  we get
  \begin{align}
    p \left(\phi_p^y(y),y\right)
      & = p \left(\phi_f^y(y),y\right)
        + p_x\left(\phi_p^y(y)-\phi_f^y(y)\right)
        + \frac{p_{xx}}{2}\left(\phi_p^y(y)-\phi_f^y(y)\right)^2
    \label{eq:expansionOfpEq1}\\
      & = p \left(\phi_f^y(y),y\right) + \left(\phi_p^y(y)-\phi_f^y(y)\right)
        \left[a\left(\phi_p^y(y) + \phi_f^y(y)\right) + by + d\right],
    \label{eq:expansionOfpEq2}
  \end{align}
  where $a$, $b$, and $d$ are the coefficients of the $x^2$, $xy$,
  and $x$ terms in $p(x,y)$ respectively.
  According to (\ref{eq:estimateErrorOfLines}),
  (\ref{eq:expansionOfpEq1}), (\ref{eq:expansionOfpEq2})
  and (\ref{eq:fittingCoef}),
  we deduce that $|\phi_f^y - \phi_p^y| = O(h^3)$.
  A similar analysis yields $|\phi_f^x - \phi_p^x| = O(h^3)$.
  Hence, we have
  \begin{equation}\label{eq:err2Estimation}
    err_2 \le O(1)\cdot \|S_{D_f\oplus D_p}\| \le O(h^4).
  \end{equation}
  Consequently, we conclude $\|S_f\| = \|S_p\| + O(h^4)$
  by (\ref{eq:err1Estimation}) and (\ref{eq:err2Estimation}).
\end{proof}

\begin{figure}[H]
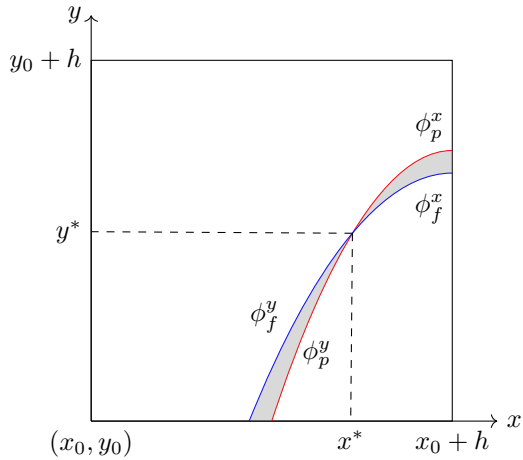

  \centering
  \includestandalone{{\TIKZDIR}FigofProoftoThmAreaError}
  \caption{The intersection lines of surfaces with the plane $z = z_0+h$.}
  \label{fig:figureForThmLeastSquaresSurfaceError}
\end{figure}

\begin{corollary}
  \label{coro:surfaceIntegralError}
  Suppose $g(x,y,z)$ and its first-order partial derivatives are bounded in $\Omega$.
  Then the surface integral error satisfies
  \begin{equation}
    \label{eq:surfaceIntegralError}
    \int_{S_f}g\mathrm{d}S - \int_{S_p}g\mathrm{d}S = O(h^4),
  \end{equation}
  and the surface-averaged error satisfies
  \begin{equation*}
    \frac{1}{\|S_f\|}\int_{S_f}g\mathrm{d}S
    - \frac{1}{\|S_p\|}\int_{S_p}g\mathrm{d}S = O(h^3).
  \end{equation*}
\end{corollary}

\begin{proof}
  Direct calculation yields
  \begin{align*}
    &\left|\int_{S_f}g\mathrm{d}S - \int_{S_p}g\mathrm{d}S\right|\\
    = & \left|\int_{D_f}g\sqrt{1+f_x^2+f_y^2}\mathrm{d}x\mathrm{d}y
        - \int_{D_p}g\sqrt{1+p_x^2+p_y^2}\mathrm{d}x\mathrm{d}y\right|\\
    \le&\left|\int_{D_f\cap D_p} g
         \frac{f_x^2+f_y^2 - p_x^2 -p_y^2}
         {\sqrt{1+f_x^2+f_y^2}+\sqrt{1+p_x^2+p_y^2}}
         \mathrm{d}x\mathrm{d}y\right|
         + \left|\int_{D_f\oplus D_p} g\cdot O(1)
         \mathrm{d}x\mathrm{d}y\right|\\
    =& O(h^4),
  \end{align*}
  where the last step follows from the proof of Theorem
  \ref{thm:leastSquaresSurfaceAreaError}.

  Let $\mathcal{C}_{\mathbf{i}}$ denote the cut cell to which
  $S_f, S_p$ belong.
  Given a point $(x_0, y_0, z_0) \in \mathcal{C}_{\mathbf{i}}$,
  applying the Taylor expansion of $g(x,y,z)$ at $(x_0, y_0, z_0)$ yields
  \begin{equation*}
    g(x,y,z) = g(x_0, y_0, z_0) + \ell(x,y,z),
  \end{equation*}
  where $\ell(x,y,z)$ represents the higher-order terms.
  According to the properties of the Taylor expansion and
  (\ref{eq:surfaceIntegralError}),
  we have
  \begin{equation}
    \label{eq:taylorOfg}
    \int_{S_f} \ell(x,y,z) \mathrm{d} S - \int_{S_p} \ell(x,y,z)\mathrm{d} S = O(h^5).
  \end{equation}
  Since $\|S_f\|,\|S_p\| = O(h^2)$,
  it follows that
  \begin{align*}
    &\frac{1}{\|S_f\|} \int_{S_f} g\mathrm{d}S - \frac{1}{\|S_p\|} \int_{S_p}g \mathrm{d}S\\
    =&\frac{1}{\|S_f\|} \int_{S_f} \ell \mathrm{d}S - \frac{1}{\|S_p\|} \int_{S_p}\ell\mathrm{d}S\\
    =&\frac{1}{\|S_f\|\|S_p\|}
       \left[\left(\|S_p\|-\|S_f\|\right)\int_{S_f} \ell\mathrm{d}S
       +\|S_f\|\left(\int_{S_f} \ell\mathrm{d}S-\int_{S_p}
       \ell\mathrm{d}S\right) \right]\\
    =&O(h^3),
  \end{align*}
  where the last step follows from (\ref{eq:leastSquaresSurfaceAreaError})
  and (\ref{eq:taylorOfg}).

\end{proof}

\begin{theorem}
  \label{thm:leastSquaresVolumeError}
  The volume error of $V_f$ and $V_p$ is
  \begin{equation*}
    \|V_f\| = \|V_p\| + O(h^5).
  \end{equation*}
\end{theorem}

\begin{proof}
  \begin{equation}\label{eq:volumeError}
    \|V_f-V_p\|\le\|V_{f}\oplus V_{p}\|
    \le\int_{D_f \cup D_p}|f(x,y) - p(x,y)|\mathrm{d}x\mathrm{d}y = O(h^5),
  \end{equation}
  where the last step follows from Corollary \ref{coro:leastSquaresFittingError2D}.
\end{proof}

\begin{corollary}
  Suppose $g(x,y,z)$ and its first-order partial derivatives are bounded in $\Omega$.
  Then we have the volume integral error
  \begin{equation*}
    \int_{V_f}g\mathrm{d}V - \int_{V_p}g\mathrm{d}V = O(h^5),
  \end{equation*}
  and the volume-averaged error
  \begin{equation}
    \label{eq:volumeIntegralAverageError}
    \frac{1}{\|V_f\|}\int_{V_f}g\mathrm{d}V - \frac{1}{\|V_p\|}\int_{V_p}g
    \mathrm{d}V = O(h^3).
  \end{equation}
\end{corollary}

\begin{proof}
  We have
  \begin{equation*}
    \left|\int_{V_f}g\mathrm{d}V - \int_{V_p}g\mathrm{d}V\right|\le
    \int_{V_f\oplus V_p}|g|\mathrm{d}V = O(h^5),
  \end{equation*}
  where the last step follows from $(\ref{eq:volumeError})$.
  And by applying similar reasoning as in the proof of
  Corollary \ref{coro:surfaceIntegralError},
  we obtain (\ref{eq:volumeIntegralAverageError}).
\end{proof}

Numerical experiments on geometric accuracy
are presented in Section \ref{sec:geometry-accuracy-tests},
which validate our theoretical results.
Furthermore, adaptive techniques can be employed
to locally enhance the mesh resolution near the boundary regions,
ensuring the desired approximation accuracy is achieved.

\subsection{Numerical Cubature}
\label{sec:numerical-cubature}

In finite volume method, it is essential to compute
integrals of a given function $f$
over a control volume $\mathcal{C} \in \mathbb{Y}$
or one of its boundary surfaces $S \subset \partial\mathcal{C}$.

For integrals over control volumes,
they can be transformed into a sum of
integrals over surfaces by the divergence theorem, i.e.,
\begin{equation}
  \label{eq:volumeIntegralToSurfaceIntegral}
  \iiint_{\mathcal{C}} f \mathrm{d} V
  = \oiint_{\partial\mathcal{C}} \mathbf{F} \cdot \mathbf{n} \mathrm{d} S,
\end{equation}
where $\mathbf{n}$ denotes the unit outward normal vector
and $\mathbf{F}$ is defined as
\begin{equation*}
  \mathbf{F}=\left(\int_{\xi_0}^{x}f(\xi,y,z)\mathrm{d}\xi,0,0\right),
\end{equation*}
with $\xi_0$ being an arbitrarily chosen real number.
For a boundary surface $S \subset \partial\mathcal{C}$
with analytic representation $\omega=\omega(u,v)$,
the right side of
(\ref{eq:volumeIntegralToSurfaceIntegral})
can be expressed as a sum of the integrals over $S$:
\begin{equation*}
  \iint_{S} \mathbf{F} \cdot \mathbf{n} \mathrm{d} S
  = \iint_{D_{uv}} (\mathbf{F}\cdot \mathbf{n})
  \sqrt{1 + \omega^2_u + \omega^2_v} \mathrm{d}u \mathrm{d}v,
\end{equation*}
where $D_{uv}$ denotes the projection of $S$ onto the $u,v$ plane.

For integrals over surfaces,
let $\mathbf{x}=(u(t),v(t)), t \in [0,1]$
be a smooth parametrization of $\partial D_{uv}$.
Given a function $g$, the application of the Green's formula yields
\begin{align}
  \iint_S g \mathrm{d} S
  &= \iint_{D_{uv}} g\big( u,v,\omega(u,v)\big)
    \sqrt{1 + \omega^2_u + \omega^2_v}
    \mathrm{d}u \mathrm{d}v \nonumber \\
  &= \iint_{D_{uv}} h(u,v) \mathrm{d}u \mathrm{d}v
    = \oint_{\partial D_{uv}} H(u,v) \mathrm{d} v \nonumber \\
  \label{eq:surfaceIntegralOfgb}
  &= \int_{0}^{1} H\big(u(t),v(t)\big) v'(t) \mathrm{d} t,
\end{align}
where $h(u,v)=g\big(u,v,\omega(u,v)\big)\sqrt{1 + \omega^2_u + \omega^2_v}$
and $H(u,v)$ is the primitive of $h(u,v)$ with respect to $u$, given by
\begin{equation*}
  H(u,v) = \int_{\xi_0}^{u} h(\xi,v)\mathrm{d} u.
\end{equation*}

The integral in (\ref{eq:surfaceIntegralOfgb})
can then be evaluated recursively
using one-dimensional numerical schemes like Gauss-Legendre quadrature.
If $\partial D_{uv}$ is merely piecewise smooth,
(\ref{eq:surfaceIntegralOfgb}) is applied to each smooth segment and
the results are aggregated.


\section{Spatial Discretization}
\label{sec:spatial-discretization}

\subsection{Poised Lattice Generation}
\label{sec:poised-lattice-generation}

Traditional finite difference (FD) methods
encounter limitations
when applied to irregular or complex geometries.
This is principally due to the fact that
FD formulas typically assume regular evenly spaced points,
and approximate the spatial derivatives
by using one-dimensional FD formulas or their tensor-product counterparts.
To address these challenges,
the poised lattice generation (PLG) algorithm was introduced \cite{PLG},
specifically designed to
generate poised lattices within complex geometries.
With the establishment of these interpolation lattices,
high-order discretization of the differential operators
becomes feasible
through the application of multivariate polynomial fitting.

Denote the first $n+1$ natural numbers by
\begin{equation*}
  \mathbb{Z}_n := \{0,1,\cdots,n\},
\end{equation*}
and the first $n$ positive integers by
\begin{equation*}
  \mathbb{Z}_n^+ := \{1,2,\cdots,n\}.
\end{equation*}

\begin{definition}
  [Lagrange interpolation problem, c.f. \cite{carnicer2006interpolation}]
  Denote by $\Pi_n^D$ the vector space of all D-variate polynomials
  of degree no more than $n$ with real coefficients.
  Given a finite number of points
  $\mathbf{x}_1,\mathbf{x}_2,\cdots,\mathbf{x}_N \in \mathbb{R}^D$,
  and the same number of data
  $f_1,f_2,\cdots,f_N \in \mathbb{R}$,
  the \emph{Lagrange interpolation problem} seeks a polynomial
  $f \in \Pi_n^D$ such that
  \begin{equation}
    \label{eq:lagrangeInterpolationProblem}
    f(\mathbf{x}_j) = f_j, \quad \forall j = 1,2,\cdots,N,
  \end{equation}
  where $\Pi_n^D$ is the \emph{interpolation space} and
  $\mathbf{x}_j$'s are the \emph{interpolation sites}.
\end{definition}

The sites $\{\mathbf{x}_j\}_{j = 1}^N$ are said to be \emph{poised}
in $\Pi_n^D$ if there exists a unique $f \in \Pi_n^D$
satisfying (\ref{eq:lagrangeInterpolationProblem})
for any given data $\{f_j\}_{j = 1}^N$.
The principal objective of the PLG algorithm is to find
poised sites near a given site in complex geometries.
In practice,
the poised sites can be arranged into
the form of triangular lattice.

\begin{definition}[Triangular lattice]
  \label{def:triangularLattice}
  A subset $\mathcal{T}^D_n$ of $\mathbb{R}^D$ is called
  \emph{a triangular lattice of degree $n$ in $D$ dimensions}
  if there exist $n+1$ distinct coordinates
  and a numbering of these coordinates,
  \begin{equation*}
    \left[
      \begin{array}{cccc}
        p_{1,0}&p_{1,1}&\cdots&p_{1,n}\\
        p_{2,0}&p_{2,1}&\cdots&p_{2,n} \\
        \vdots &\vdots &\ddots&\vdots \\
        p_{D,0}&p_{D,1}&\cdots&p_{D,n}
      \end{array}
    \right] \in \mathbb{R}^{D \times(n+1)},
  \end{equation*}
  such that $\mathcal{T}_n^D$ can be expressed as
  \begin{equation*}
    \mathcal{T}_n^D = \left\{
      (p_{1,k_1}, p_{2,k_2},\cdots, p_{D,k_D}) \in \mathbb{R}^D
      : k_i \in \mathbb{Z}_n; \sum\limits_{i = 1}^D k_i \leq n
    \right\},
  \end{equation*}
  where $p_{i,j}$ denotes the $j$th coordinate of the $i$th variable $p_i$.
\end{definition}

In \cite{PLG}, it is proved that any triangular lattice $\mathcal{T}_n^D$ is
poised in $\Pi_n^D$.
The PLG problem is to seek a collection of such triangular lattices
from available candidate points.

\begin{definition}[PLG problem]
  \label{def:PLGProblem}
  Denote the $D$-dimensional cube of size $n+1$ as
  \begin{equation*}
    \mathbb{Z}_n^D := (\mathbb{Z}_n)^D = \{0,1,\cdots,n\}^D,
  \end{equation*}
  and define the set of all triangular lattices of degree $n$ in $\mathbb{Z}_n^D$ as
  \begin{equation*}
    \mathcal{X} := \{\mathcal{T}_n^D: \mathcal{T}_n^D \subset \mathbb{Z}_n^D\}.
  \end{equation*}
  For a set of feasible nodes $K \subseteq \mathbb{Z}_n^D$
  and a starting point $\mathbf{q} \in K$,
  the \emph{PLG problem} seeks
  $\mathcal{T} \in \mathcal{X}$ such that
  $\mathbf{q} \in \mathcal{T}$
  and $\mathcal{T} \subseteq K$.
\end{definition}

PLG algorithm solves the PLG problem by back-tracking.
More details can be found in \cite{PLG}.

\begin{figure}[htp]
  \centering
  \includegraphics[width=0.8\textwidth]{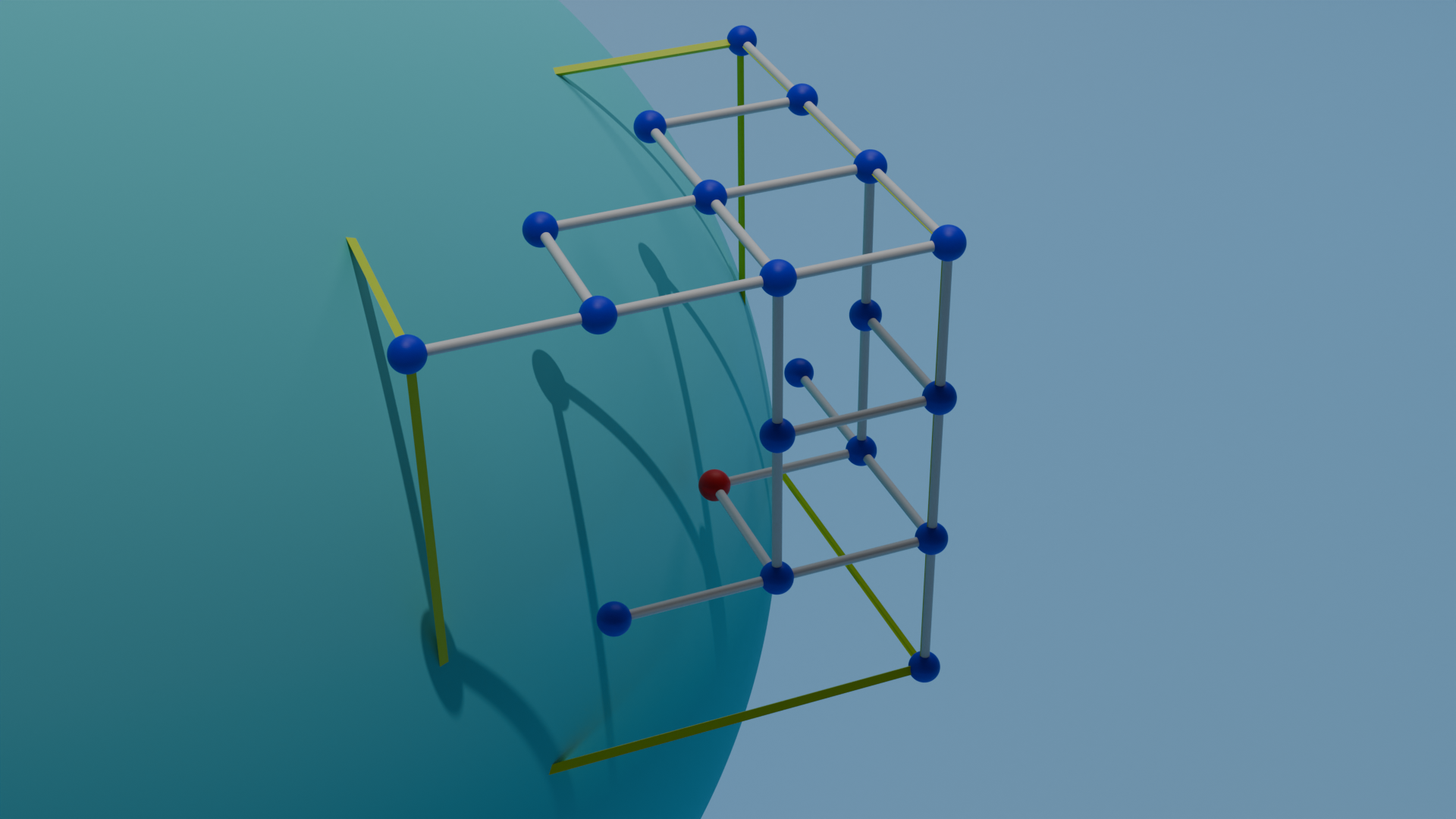}
  \caption{For the finite-difference discretization of a spatial
    operator at red FD node $\mathbf{x}_j$,
    we select a poised lattice
    $\mathcal{T}_{\mathbf{j}} = \{\mathbf{x}_j\}$
    in $\Pi_3^3$.
    The red node and the blue nodes represent
    $\mathcal{T}_{\mathbf{j}}$ and the
    ellipsoid represents the irregular boundary.}
  \label{fig:plg3d}
\end{figure}

\subsection{Merging Algorithms}
\label{sec:merging-algorithm}

\begin{definition}
  \label{def:thetaProper}
  A cut cell $\mathcal{C}_{\mathbf{i}}$
  is called a \emph{$\theta$-proper cell}
  if it is non-empty, connected and satisfies
  \begin{equation*}
    \frac{\|\mathcal{C}_{\mathbf{i}}\|}{h^D} \ge \theta,
  \end{equation*}
  where $D=3$, $h \in \mathbb{R}^+$ is the spacing of the grid,
  and $\theta \in (0,\frac{1}{2})$ is a user-defined tolerance.
\end{definition}

To ensure the robustness of our method,
it is necessary to merge cells that are not $\theta$-proper.

A cut cell $\mathcal{C}_{\mathbf{i}}$
is called \emph{multi-component}
if it contains more than one connected component.
It can be represented as
$\mathcal{C}_{\mathbf{i}}=\bigcup_{k=1}^{n_c} {\cal C}_{\mathbf{i}}^{k}$,
where $n_c>1$ indicates the number of components,
and $\mathcal{C}_{\mathbf{i}}^k$'s are pairwise distinct.
In particular,
if $\mathcal{C}_{\mathbf{i}}$ does not consist of multiple components,
it is understood that
$\mathcal{C}_{\mathbf{i}}=\mathcal{C}_{\mathbf{i}}^1$.
Let $\hat{\mathcal{C}}_{\mathbf{i}}
(\text{or} \hat{\mathcal{C}}_{\mathbf{i}}^k)$
denote the union of those cells that are merged with
$\mathcal{C}_{\mathbf{i}}\ (\text{or}\ \mathcal{C}_{\mathbf{i}}^k)$,
including itself.
If no cells are merged with $\mathcal{C}_{\mathbf{i}}$,
then $\mathcal{C}_{\mathbf{i}} = \hat{\mathcal{C}}_{\mathbf{i}}$.
Moreover, to represent the grid structure,
we construct an undirected graph $G=(V,E)$,
where each vertex $v \in V$ is associated with
a cell component $\mathcal{C}_{\mathbf{i}}^k$,
and an edge $e \in E$ connects any two components,
$\mathcal{C}_{\mathbf{i}}^k$ and $\mathcal{C}_{\mathbf{j}}^{k'}$,
that share a common face.

We design Algorithm \ref{alg:cellMerging} with the
following core merging principles:

\begin{enumerate}[label = (MAP-\arabic{enumi})]
  \item Two cut cells
  $\mathcal{C}_{\mathbf{i}}$ and $\mathcal{C}_{\mathbf{j}}$ are \emph{mergeable}
  if they share a common face
  and satisfy one of the following conditions:
  (a) neither cell is multi-component,
  and at least one of them is $\theta$-proper;
  (b) one cell is multi-component,
  while the other is a non-multi-component $\theta$-proper cell.
  \label{map:mergeable}

  \item For a multi-component cell
  $\mathcal{C}_{\mathbf{i}}
  =\bigcup_{k=1}^{n_c} {\cal C}_{\mathbf{i}}^{k}$ ($n_c\ge2$),
  we merge each component with its adjacent mergeable cell.
  For each $\mathcal{C}_{\mathbf{i}}^{k}$, we
  select an adjacent cell $\mathcal{C}_{\mathbf{j}}$
  such that the area of their common face
  is the largest among all its mergeable cells.
  Then, $\mathcal{C}_{\mathbf{i}}^{k}$ is absorbed into this
  neighboring cell via
  \begin{equation*}
    \hat{{\cal C}}_{\mathbf{j}}\leftarrow\hat{{\cal C}}_{\mathbf{j}}
    \cup^{\bot \bot }\hat{{\cal C}}_{\mathbf{i}}^{k},
  \end{equation*}
  as shown in Figure \ref{fig:multi-component}.
  \label{map:multi-component}

  \item For a non-multi-component cell $\mathcal{C}_{\mathbf{i}}$
  with $\|\mathcal{C}_{\mathbf{i}}\|<\theta h^D$,
  we select an adjacent cell $\mathcal{C}_{\mathbf{j}}$
  such that the area of their common face
  is the largest among all its mergeable cells.
  Subsequently, $\mathcal{C}_{\mathbf{i}}$ is absorbed into this neighbor via
  \begin{equation*}
    \hat{{\cal C}}_{\mathbf{j}}\leftarrow\hat{{\cal C}}_{\mathbf{j}}
    \cup^{\bot \bot }\hat{{\cal C}}_{\mathbf{i}},
  \end{equation*}
  as shown in Figure \ref{fig:mergable}.
  \label{map:smallCells}
\end{enumerate}

\begin{algorithm}
	\caption{\textbf{CellMerging}}
	\begin{algorithmic}[1] \label{alg:cellMerging}
		\REQUIRE The computational domain $\Omega \in \mathbb{Y}$, \\
    the grid width $h<(\|\Omega\|)^{\frac{1}{3}}$, \\
    the user-specified threshold $\theta \in (0, \frac{1}{2})$.
    \ENSURE A set $\{\hat{\mathcal{C}}\}$ of merged cells.
    \PreCondition There is at least one non-multi-component cell in $\Omega$.
		\PostCondition All multi-component cells have been merged. \\
    For any non-multi-component cell ${\cal C}_{\mathbf{i}}$,
    $\hat{{\cal C}}_{\mathbf{i}}$ is $\theta$-proper.\\
    \STATE Initialize $\mathcal{M}_{out}$ as the
    set of cells generated by embedding $\Omega$ into the Cartesian grid:
    $\mathcal{M}_{out} \leftarrow \{ \mathcal{C}_{\mathbf{i}} = C_{\mathbf{i}} \cap \Omega\}$.
    \STATE Preprocess all multi-component cells in $\mathcal{M}_{out}$
      according to \ref{map:multi-component}.
    \STATE Process all cells in $\mathcal{M}_{out}$
      according to \ref{map:smallCells}.
    \FOR{each $\mathcal{C}_{\mathbf{i}}\in\mathcal{M}_{out}$ with
      $\|\hat{\mathcal{C}}_{\mathbf{i}}^k\|< \theta h^D$ or
      each multi-component cell $\mathcal{C}_{\mathbf{i}}\in\mathcal{M}_{out}$
      with component $\mathcal{C}_{\mathbf{i}}^k$ unmerged}
    \STATE Let $S$ denote the set of cell components,
      generated by performing a Breadth-First Search (BFS)
      on graph $G(\mathcal{M}_{out})$
      starting from $\mathcal{C}_{\mathbf{i}}^k$.
    \FOR{each $\mathcal{C}_{\mathbf{j}}^{k'} \in S$}
      \STATE $\hat{{\cal C}}_{\mathbf{i}}^k\leftarrow\hat{{\cal C}}_{\mathbf{i}}^{k}\
        \cup^{\bot \bot }\hat{{\cal C}}_{\mathbf{j}}^{k'}.$
      \IF{$\|\hat{\mathcal{C}}_{\mathbf{i}}^k\|\geq \theta h^D$}
      \BREAK.
      \ENDIF
    \ENDFOR
    \ENDFOR
	\end{algorithmic}
\end{algorithm}

\begin{figure}[htp]
  \centering
  \begin{subfigure}[t]{0.48\textwidth}
    \centering
    \includegraphics[width=1.0\textwidth]{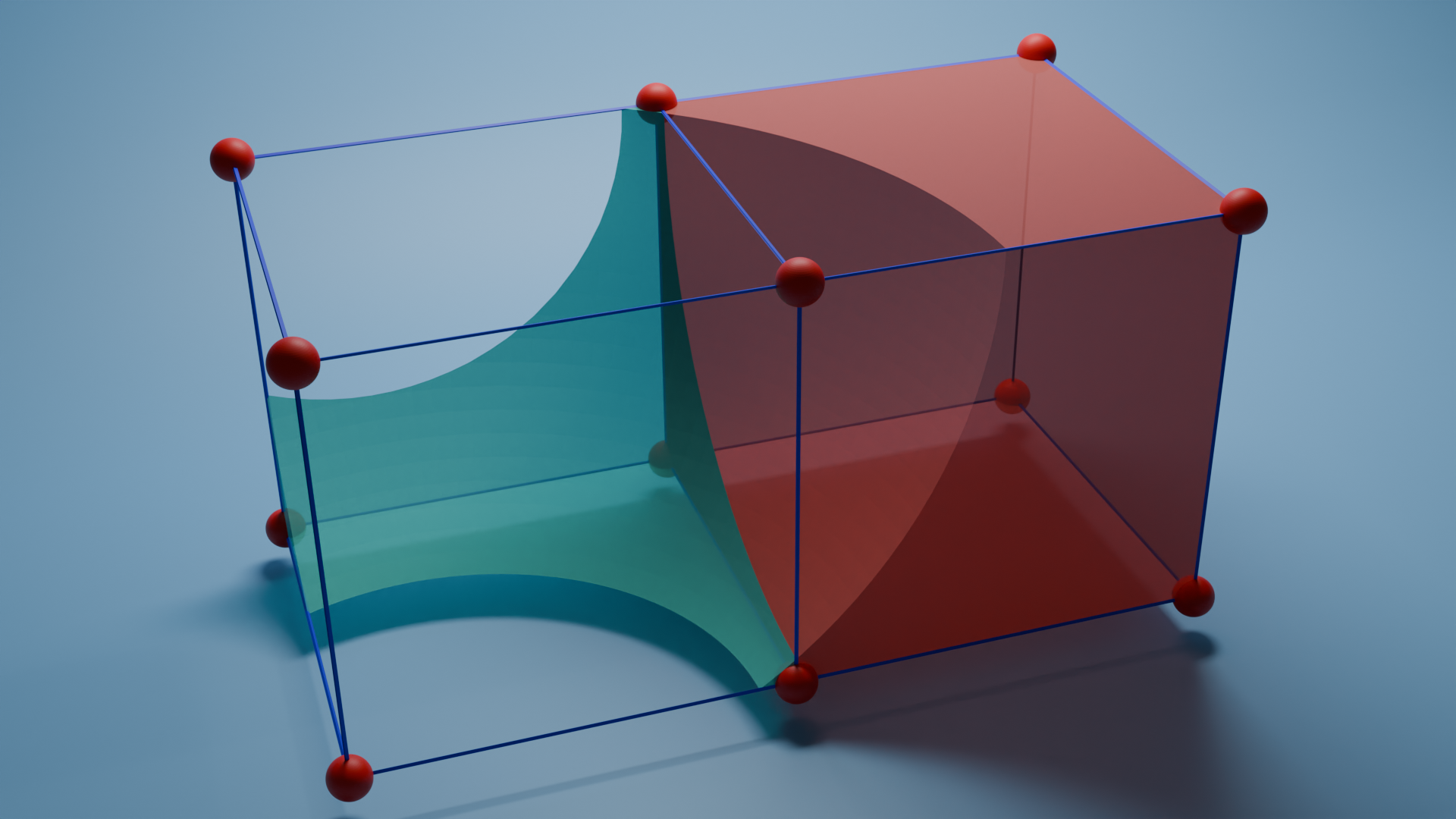}
    \caption{
      Merging of single small cells: the cyan cut cell
      will be merged into the red one.
    }
    \label{fig:mergable}
  \end{subfigure}
  \begin{subfigure}[t]{0.48\textwidth}
    \centering
    \includegraphics[width=1.0\textwidth]{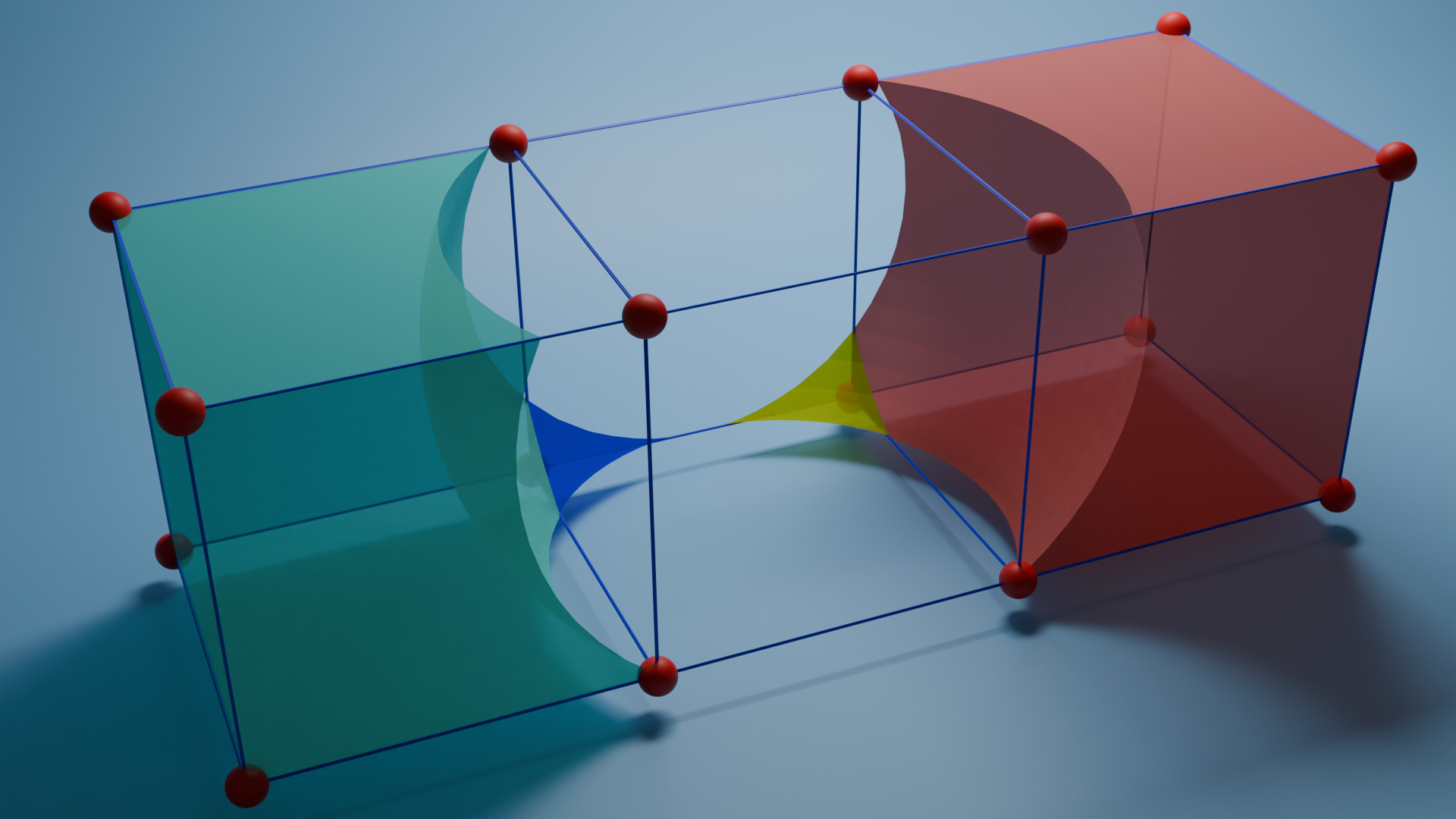}
    \caption{
      Merging of multi-component cells:
      the middle cell consists of two components.
      The blue component will be merged into the cyan cut cell,
      while the yellow component will be merged into the red one.
    }
    \label{fig:multi-component}
  \end{subfigure}
  \caption{Illustration of cell merging.}
  \label{fig:illustrationOfCellMerging}
\end{figure}

Algorithm~\ref{alg:cellMerging} operates in two main steps.
First, it processes all multi-component cells
and small cut cells according to the criteria outlined in
\ref{map:multi-component} and \ref{map:smallCells},
respectively.
This step merges nearly all multi-component cells and small cut cells.
Next, for any remaining non-$\theta$-proper cell
or unmerged multi-component cell, a Breadth-First Search (BFS) is performed
on the graph $G(\mathcal{M}_{out})$ starting from it.
During the traversal,
the cell is incrementally merged with its neighboring cells
until it satisfies the $\theta$-proper condition.
Since the domain $\Omega$ is connected,
its corresponding graph $G(\mathcal{M}_{out})$ is also connected,
guaranteeing the successful and
efficient merging of all multi-component cells
and small cut cells by Algorithm~\ref{alg:cellMerging}.


\section{Multigrid}
\label{sec:multigrid}

In this section, we present a modified multigrid solver for
solving (\ref{eq:multigridLinearSystem}).
In our modified multigrid algorithm,
the smoother operator is coupled with LU factorization
\cite{brandt2011multigrid},
a technique we refer to as "LU-correction",
with $O(\frac{1}{h^2})$ unknowns.
Traditional LU factorization results in a complexity of $O(\frac{1}{h^6})$.
However, owing to the sparsity of the matrix,
avoiding explicit manipulation of zeros
can lead to substantial computational time savings.
We have proved that the complexity of the LU-correction
can be reduced to $O(\frac{1}{h^3})$ by employing
the nested dissection (ND) ordering,
allowing a full multigrid method (FMG) with optimal complexity.

\subsection{Nested Dissection Ordering}

Consider solving a sparse linear system
\begin{equation*}
  Ax = b
\end{equation*}
by LU factorization,
where $A$ is an $n \times n$ sparse symmetric matrix
that can be decomposed as $A=LU$.
Avoiding explicit operations on zeros can
significantly reduce computation time.
However, the process of LU factorization often
introduces new nonzero elements,
known as \emph{fill-ins},
in positions where $A$ originally had zeros.
These fill-ins can
greatly affect the computational efficiency.
To minimize fill-ins,
an effective strategy is
to permute the rows and columns of $A$.
This transformation can be represented as:
\begin{equation*}
  A^{\prime}=PAP^T,
\end{equation*}
where $P$ is a permutation matrix.
By solving the reordered system,
the sparsity of the matrix can be better preserved.

A symmetric matrix $A$ can be represented by an undirected graph
$G = (V,E)$. The graph $G$ contains one vertex $i\in V$
for each row (and column) in $A$, and one edge $\{i,j\}\in E$ for
each pair of nonzero, off-diagonal elements $a_{ij} = a_{ji} \ne 0$ in $A$.
In particular, for partial differential equations involving
one physical unknown per mesh point, the adjacency graph of the matrix
arising from the discretization is often the graph
represented by the mesh itself.
Each permutation matrix $P$ corresponds to
a numbering of the vertices of $G$, i.e., to a one-to-one mapping
$\pi : V \rightarrow \{1,2,\cdots,n\}$.

\begin{lemma}
  \label{lem:LUFactoringComplexity}
  For a sparse symmetric matrix $A\in\mathbb{R}^{n\times n}$, when
  operations on zeros are avoided and pivoting is not employed,
  the total number of operations required for its
  LU factorization is given by
  \begin{equation}
    \label{eq:operationCountOfLU}
    \zeta = \sum_{k=1}^{n-1}2\nu_k(\nu_k + 1),
  \end{equation}
  where $\nu_k$ denotes the number of nonzero elements
  excluding the diagonal in the $k$-th row at the $k$-th step of
  the Gaussian elimination.
\end{lemma}

The ND ordering
\cite{george1973nested,karypis1998fast,
  lipton1979generalized,saad2003iterative}
is primarily used to reduce fill-ins by providing an effective
mapping $\pi$ of a given graph $G$.
This technique is described
by recursively finding separators in the graph,
as shown in Algorithm \ref{alg:ND}.
A set $S$ of vertices in a graph is called a \emph{separator}
if its removal splits the graph into two
disjoint subgraphs.
The main step of the ND procedure involves
partitioning the graph into three parts:
two disjoint subgraphs and a separator that disconnects them.
In Algorithm \ref{alg:ND},
the numbering is performed in reverse order,
starting from the highest to the lowest.
This ensures that at each level,
the rows (and columns) corresponding to the separator vertices
are eliminated last.
An example illustrating this process is shown in
Figure \ref{fig:anIllustrativeExampleOfNDIn2DGrid}.
Actually, the ND ordering method aims to control the size of $\nu_k$
in (\ref{eq:operationCountOfLU}) through the independence
between subgraphs at each step.
Figure \ref{fig:Dissect} demonstrates the
application of ND ordering in our problem,
significantly reducing the number
of fill-ins during Gaussian elimination.

\begin{algorithm}
	\caption{ND($G$, $a_{\text{min}}$)}
  \label{alg:ND}
	\begin{algorithmic}[1]
    \REQUIRE Graph $G = (V, E)$;
    minimum number of vertices to split $a_{\text{min}}$;
    \SideEffect Vertices in $V$ have a new numbering.
    \IF{$|V| \leq a_{\text{min}}$}
    \STATE Number the vertices in $V$.
    \ELSE
    \STATE Find a separator $S$ for $V$.
    \STATE Number the vertices in $S$.
    \STATE Split $V$ into $G_L, G_R$ by removing $S$.
    \STATE ND($G_L$, $a_{\text{min}}$).
    \STATE ND($G_R$, $a_{\text{min}}$).
    \ENDIF
	\end{algorithmic}
\end{algorithm}

\begin{figure}[htp]
  \centering
  \begin{subfigure}[t]{0.48\textwidth}
    \centering
    \includestandalone[width=.78\textwidth]{{\TIKZDIR}PartitionOfGrid}
    \caption{
      Partition of a two-dimensional regular domain.
    }
    \label{fig:NDGridPartition}
  \end{subfigure}
  \begin{subfigure}[t]{0.48\textwidth}
    \centering
    \includegraphics[width=1.0\textwidth]{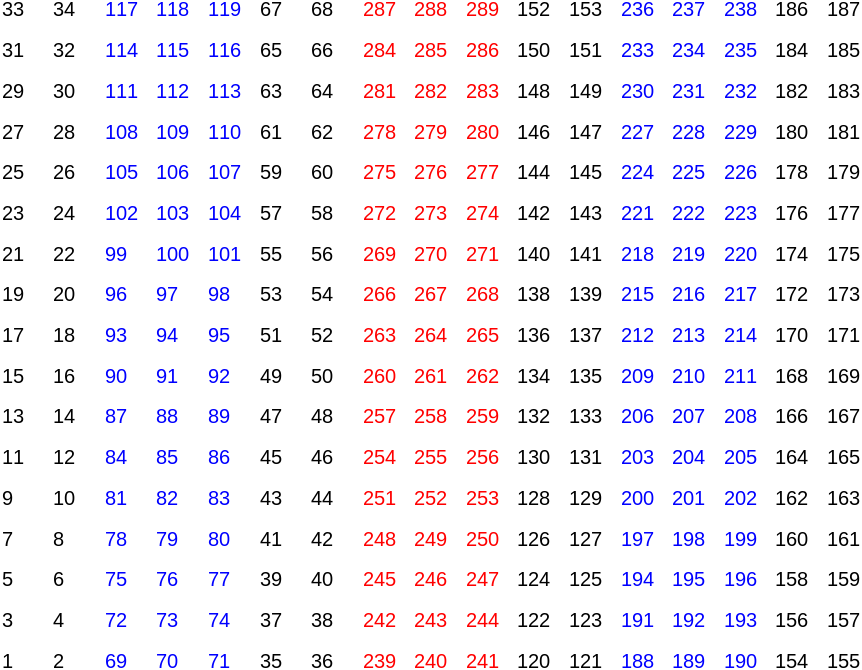}
    \caption{Nested Dissection ordering of (a).}
    \label{fig:orderingOf2DGrid}
  \end{subfigure}
  \caption{
      (a) illustrates the partition of a two-dimensional regular domain
      grid using the finite volume method,
      where the stencil of cell $\mathbf{i}$ includes its
      three adjacent cell layers $\{\mathbf{i} \pm \mathbf{e}^d, \mathbf{i}
      \pm 2\mathbf{e}^{d}, \mathbf{i} \pm 3\mathbf{e}^{d}, d = 1,2\}$.
      In the initial recursion step $C = P_3$, 
      with $A$ occupying the left part and $B$ the right, respectively.
      The second recursion assigns
      $C = P_2$, $A = P_1$.
      (b) shows the corresponding ordering.
  }
  \label{fig:anIllustrativeExampleOfNDIn2DGrid}
\end{figure}

\subsection{A Specific ND Ordering Algorithm}
\label{sec:nested-dissection-ordering}

\begin{definition}
  \label{def:f(n)SeparatorCondition}
  Let ${\cal S}$ be a class of graphs closed under the subgraph relation
  (i.e., if $G_2 \in {\cal S}$ and $G_1$ is a subgraph of $G_2$
  then $G_1 \in {\cal S}$).
  The class ${\cal S}$ satisfies an \emph{$f(n)$-separator condition}
  if there exist constants
  $\alpha \in \left[\frac{1}{2}, 1\right], \beta \in \mathbb{R}^+$,
  for any $n$-vertex subgraph $G$ of ${\cal S}$,
  the vertices of $G$ can be partitioned into three sets $A,B,C$,
  such that no vertex in $A$ is adjacent to any vertex in $B$,
  $|A|, |B| \leq \alpha n$ and $|C| \leq \beta f(n)$,
  where $f(n)$ is a given function of $n$.
\end{definition}

For an $n$-vertex graph $G$ belonging to a family of graphs ${\cal S}$
that satisfies the
$\sqrt{n}$-separator condition,
a specific ND ordering algorithm is
detailed in Algorithm \ref{alg:NDOrder}.
The impact of this ordering on the LU factorization
is described by the two theorems presented below.
By employing Algorithm \ref{alg:NDOrder},
the LU factorization of the matrix corresponding to $G$
exhibits a complexity of $O(n^{\frac{3}{2}})$.

\begin{algorithm}
	\caption{NDOrder($G$, $a$, $b$)}
  \label{alg:NDOrder}
	\begin{algorithmic}[1]
    \REQUIRE Graph $G = (V, E)$;
    start number $a$; end number $b$;
    constants $\alpha, \beta$;
    \SideEffect Vertices in $V$ have a new numbering from $a$ to $b$.
    \IF{$|V| \leq \frac{\beta}{(1-\alpha)^2}$}
    \STATE Number the unnumbered vertices arbitrarily from $a$ to $b$.
    \ELSE
    \STATE $n\leftarrow |V|$.
    \STATE Find sets $A,B,C \subset V$ satisfying
    the $\sqrt{n}$-separator condition.
    \STATE Number the unnumbered vertices in $C$ arbitrarily
    from $b-|C|+1$ to $b$.
    \STATE NDOrder($B \cup C$, $b-|B|-|C|+1$, $b-|C|$).
    \STATE NDOrder($A \cup C$, $a$, $a+|A|-1$).
    \ENDIF
	\end{algorithmic}
\end{algorithm}

\begin{theorem}[Lipton et al. \cite{lipton1979generalized}]
  Let $G$ be any $n$-vertex graph numbered
  by Algorithm \ref{alg:NDOrder},
  the total size of the fill-in in LU factorization
  associated with the numbering is at most $c_1n\log_2n + O(n)$,
  where
  \begin{equation*}
    c_1= - \frac{\beta^2(1 + 3 \sqrt{\alpha})}
                {2(1 - \sqrt{\alpha})\log_2 \alpha}.
  \end{equation*}
\end{theorem}

\begin{theorem}[Lipton et al. \cite{lipton1979generalized}]
  \label{thm:multiplicationCountOfNDSortedGraph}
  Let $G$ be any $n$-vertex graph numbered
  by Algorithm \ref{alg:NDOrder},
  the total multiplication count in LU factorization
  associated with the numbering is
  at most $c_2n^{\frac{3}{2}} + O(n(\log_2 n)^2)$, where
  \begin{equation*}
    c_2= \frac{\beta^2}{1 - \delta}
    \left( \frac{1}{6}+\frac{\beta \sqrt{\alpha}}{1-\sqrt{\alpha}}
      \left(2+ \frac{\sqrt{\alpha}}{1+\sqrt{\alpha}
          + \frac{4 \alpha}{1-\alpha}}\right)\right),
  \end{equation*}
  with $\delta = \alpha^{\frac{3}{2}} + (1-\alpha)^{\frac{3}{2}}$.
\end{theorem}

In addition, for a given graph $G$, multiple methods can be employed to
find such a separator $C$ in Algorithm \ref{alg:NDOrder},
including spectral partitioning methods
\cite{pothen1990partitioning,pothen1992towards},
the multilevel spectral bisection algorithm \cite{barnard1994fast},
geometric partitioning algorithms
\cite{heath1995cartesian,miller1993automatic,teng1991unified}
and multilevel graph partitioning schemes
\cite{bui1993heuristic,hendrickson1995multi,karypis1998fast}.
Research conducted in \cite{hendrickson1995multi}
demonstrates that multilevel graph partitioning schemes can yield
superior partitioning efficiency
and quality compared to alternative methods
for various finite element problems similar to the ones we are studying.
Consequently, we adopt the multilevel schemes,
which involves three phases:
reducing the size of the graph (i.e., coarsening the graph)
by collapsing vertices and edges,
partitioning the smaller graph,
and then uncoarsening it to
construct a partition for the original graph.
For each phase, there are also multiple
approaches available; see \cite{karypis1998fast}.

As for the complexity of Algorithm \ref{alg:NDOrder}
with utilizing the multilevel schemes,
for an $n$-vertex graph $G$,
we assume that the number of vertices
in the graph can be reduced at a fixed rate
during each step of the coarsening phase.
Consequently, a 2-way partitioning of
the original graph $G$ (finding the first graph separator)
requires $O(n)$ time.
For the two resulting subgraphs of $G$,
the total time for their 2-way partitioning also requires $O(n)$.
Moreover, $O(\log(n))$ recursive steps are necessary to
complete the ND ordering of $G$.
Therefore, the overall time complexity of
Algorithm \ref{alg:NDOrder} is $O(n\log(n))$.

\subsection{Multigrid Components}
\label{sec:multigrid-components}

Assume that $\Omega^{\ast} = \{\Omega^{(m)}: 0 \leq m \leq M\}$ is
a hierarchy of grids,
where $M \in \mathbb{Z}^+$ denotes the number of grids,
and $\Omega^{(m+1)} = \mathbf{Coarsen}(\Omega^{(m)})$.
The relationship between the grid spacing of
the $m$th grid and the $0$th grid
often follows $h^{(m)} = 2^mh^{(0)}$.
Practically, the total number of cells contained in
the coarsest grid $\Omega^{(M)}$ is controlled
by a fixed small upper bound,
allowing a direct linear system solver (such as LU factorization)
to be applied with minimal time consumption.
Our modified multigrid algorithm, as shown in
Algorithm \ref{alg:multigrid} and Algorithm \ref{alg:V-cycle},
employs the LU-correction to account for
the particularities of irregular domains.
The update procedure can be divided into two stages:

\begin{enumerate}[label=(SMO-\arabic*)]
\item Smoother: execute an $\omega$-weighted Jacobi iteration
  \begin{equation*}
    \hat{\varphi}_1^{\prime} = D^{-1}\left[(1-\omega)D + \omega O\right]\hat{\varphi}_1
    + \omega D^{-1}(\hat{r}_1 - L_{12}\hat{\varphi}_2),
  \end{equation*}
  where $D$ is the diagonal of $L_{11}$,
  and $O = D - L_{11}$.
  \label{smo1}
\item LU-correction:
  \begin{itemize}
  \item Derive the permutation matrix $P$
  through the application of the nested dissection ordering method
  (detailed in Section \ref{sec:nested-dissection-ordering}) to the
  symmetric matrix
  $L_{22} + L_{22}^T$,
  and denote the reordered matrix as $L_{22}^{\prime} = PL_{22}P^T$.
  \item Employ LU factorization to solve the linear system
  \begin{equation*}
    L_{22}^{\prime}\psi
    = P(\hat{r}_2 - L_{21} \hat{\varphi}_1^{\prime}),
  \end{equation*}
  and update $\hat{\varphi}_2$ by
  $\hat{\varphi}_2^{\prime} = P^T\psi$.
  \label{smo2}
  \end{itemize}
\end{enumerate}

\begin{algorithm}
	\caption{Multigrid}
  \label{alg:multigrid}
	\begin{algorithmic}[1]
    \REQUIRE Hierarchy of grids $\Omega^{\ast}$;
    the discretization operators of each grid $L^{(m)}$;
    the maximum number of iterations $I_{\max}$;
    the residual $\hat{r}$;
    the initial guess $\hat{\varphi}_g$;
    exit condition $\epsilon$.
    \ENSURE Solution for the linear system $L^{(0)}\hat{\varphi} = \hat{r}$.
    \STATE $\hat{\varphi} \leftarrow \hat{\varphi}_g$.
    \FOR{$i=1$ to $I_{\max}$}
    \STATE $\hat{s}^{(0)} \leftarrow \hat{r} - L^{(0)}\hat{\varphi}$.
    \IF {$\frac{\|\hat{s}^{(0)}\|}{\|\hat{r}\|} < \epsilon$}
    \STATE Exit the loop.
    \ENDIF
    \STATE $\hat{\varphi} \leftarrow \hat{\varphi} +
    \mathbf{VCycle}(\hat{s}^{(0)})$.
    \ENDFOR
    \RETURN $\hat{\varphi}$.
	\end{algorithmic}
\end{algorithm}

\begin{algorithm}
	\caption{VCycle}
  \label{alg:V-cycle}
	\begin{algorithmic}[1]
    \REQUIRE
    An integer $M \in \mathbb{Z}^+$ indicates the number of grid levels;
    an integer $m \in \{0,1,\cdots,M\}$ indicates the hierarchy depth;
    the discretization operator of the $m$th grid $L^{(m)}$;
    the residual of the $m$th grid $\hat{s}^{(m)}$;
    multigrid parameters $\nu_1,\nu_2$.
    \ENSURE Solution for $L^{(m)}\hat{\varphi}^{(m)} = \hat{s}^{(m)}$.
    \IF{$m = M$}
    \STATE Use bottom solver to solve the linear system
    $L^{(M)}\hat{\varphi}^{(M)} = \hat{s}^{(M)}$.
    \ELSE
    \STATE Apply smoother and LU-correction $\nu_1$ times.
    \STATE $\hat{s}^{(m+1)} \leftarrow
    \mathbf{Restrict}(\hat{s}^{(m)} - L^{(m)}\hat{\varphi}^{(m)})$.
    \STATE $\hat{\varphi}^{(m+1)} \leftarrow
    \mathbf{VCycle}(\hat{s}^{(m+1)})$.
    \STATE $\hat{\varphi}^{(m)} \leftarrow
    \hat{\varphi}^{(m)} + \mathbf{Prolong}(\hat{\varphi}^{(m+1)})$.
    \STATE Apply smoother and LU-correction $\nu_2$ times.
    \ENDIF
    \RETURN $\hat{\varphi}^{(m)}$.
	\end{algorithmic}
\end{algorithm}

In practical implementation,
it is favorable to pre-compute the permutation matrix $P$
and the LU factorization of $L_{22}^{\prime}$,
thereby avoiding repetitive executions of
LU factorization in each V-cycle iteration.
After two iterations of the smoother and LU-correction, we have
\begin{subequations}
\begin{gather}
  \hat{e}^{\prime}_1 = D^{-1}\left[(1-\omega)D + \omega
    O\right]\hat{e}_1, \label{eq:smootherResidualA} \\
  \hat{e}^{\prime}_2 = \hat{r}_2 - L_{21}\hat{\varphi}^{\prime}_1
  - L_{22}\hat{\varphi}^{\prime}_2= \mathbf{0},
  \label{eq:smootherResidualB}
\end{gather}
\end{subequations}
where $\hat{e} = [\hat{e}_1^T, \hat{e}_2^T]^T$ and its prime version
are the residuals in (\ref{eq:multigridLinearSystem}) before
and after the iteration respectively.
(\ref{eq:smootherResidualA}) illustrates that
the residuals on $\hat{\varphi}_1$ can be
well-controlled by the weighted Jacobi iteration,
while the residuals on $\hat{\varphi}_2$ are zeros after
applying the LU-correction.

Regarding the \textbf{Restrict} and \textbf{Prolong} operators,
we apply the volume weighted restriction :
\begin{equation*}
  \langle \varphi \rangle_{\lfloor \frac{\mathbf{i}}{2} \rfloor}^{(m+1)}
  = 2^{-D} \sum \limits_{\mathbf{j} \in \{0, 1\}^D}\langle
  \varphi \rangle_{\mathbf{i} + \mathbf{j}}^{(m)}
\end{equation*}
and the patch-wise constant interpolation
\begin{equation*}
  \langle \varphi \rangle_{\mathbf{i}}^{(m)}
  = \langle \varphi \rangle_{\lfloor \frac{\mathbf{i}}{2}\rfloor}^{(m+1)}
\end{equation*}
while leaving the correction and the residual
for cells in $\hat{\varphi}_2$ to zero.

At the coarsest level,
the system $L^{(M)}\hat{\varphi}^{(M)} = \hat{s}^{(M)}$
is solved using an LU solver,
with the LU factorization of $L^{(M)}$ pre-computed to optimize efficiency.

\subsection{Complexity Analysis}
\label{sec:complexity-analysis}
Here we analyze the complexity of our modified multigrid method.
The operations within Algorithm \ref{alg:V-cycle}
include application of the smoother, LU-correction,
restriction and prolongation operators on each grid.
Notably, since the cumulative complexity of 
the entire grid hierarchy is equivalent
to a constant multiple of the finest gird's complexity,
we concentrate solely on the computations on the finest grid.
Let $h = h^{(0)}$ denote the spacing of the finest grid $\Omega^{(0)}$,
and let $N = \dim \hat{\varphi}$ and $N_2 = \dim \hat{\varphi}_2$.
In three-dimensional problems,
$N = O(\frac{1}{h^3})$, $N_2 = O(\frac{1}{h^2})$ 
and $\dim \hat{\varphi}_1 = N - N_2 = O(\frac{1}{h^3})$.
\begin{itemize}
\item The \textbf{Restrict} and \textbf{Prolong} operators
  are applied to each unknown variable,
  demanding a computational cost of $O(N) = O(\frac{1}{h^3})$.
\item The Smoother \ref{smo1}
  requires $O(\frac{1}{h^3})$ computational cost 
  due to the execution of the
  $\omega$-weighted Jacobi iteration on $\hat{\varphi}_1$.
\item The LU-correction \ref{smo2} involves ND ordering and LU factorization.
  The ND ordering incurs a computational cost 
  of $O(\frac{1}{h^2}\log(\frac{1}{h}))$
  (as detailed in Section \ref{sec:nested-dissection-ordering}).
  Besides, the LU factorization of $L_{22}^{\prime}$ 
  requires $O(\frac{1}{h^3})$ cost,
  as proved below.
\end{itemize}

\begin{proposition}
  The matrix $L_{22}$ in (\ref{eq:multigridLinearSystem})
  satisfies the $\sqrt{n}$-separator
  condition, where $n=O(\frac{1}{h^2})$.
\end{proposition}

\begin{proof}
  Each row of $L_{22}$ corresponds to a cell employing discretization
  (\ref{eq:irregularLaplaceDiscretization})
  within the three-dimensional grid,
  with its nonzero entries mapping to cells in the PLG stencil.
  Since the PLG stencil is a triangular lattice 
  with $p+1$ distinct coordinates,
  the grid (or graph) can be partitioned into 
  two independent parts by a slicing of width $p$,
  where $p$ is the degree of the fitted polynomial.
  A representative example is illustrated in 
  Figure \ref{fig:plgNDSeperate} with $p=4$.
  A slice with width $p$ owns $O(\frac{1}{h})$ cells,
  while the total number of cells 
  corresponding to $L_{22}$ is $O(\frac{1}{h^2})$,
  thereby $L_{22}$ satisfies $\sqrt{n}$-separator condition 
  with $n=O(\frac{1}{h^2})$.
\end{proof}
\begin{figure}[h]
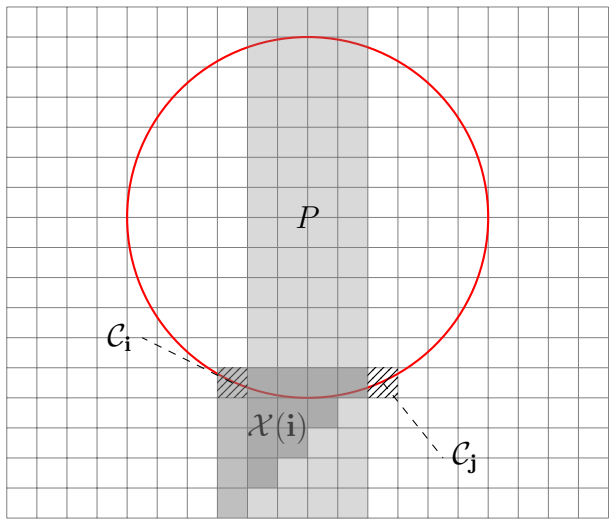

	\centering
  \includestandalone{{\TIKZDIR}PLGNDSeperate}
	\caption{Illustration of $\sqrt{n}$-separator:
    in a projection onto the $xy$-plane,
    the separator $P$, which is a split with width $4$,
    effectively isolates any cell $\mathcal{C}_{\mathbf{j}}$
    in the right-hand part from belonging to the stencil of
    any cell $\mathcal{C}_{\mathbf{i}}$ in the left-hand part
    (i.e., $\mathcal{X}(\mathbf{i})$).
    Consequently, $P$ acts as a separator dividing the domain
    into two independent regions.
  }
	\label{fig:plgNDSeperate}
\end{figure}

By Theorem \ref{thm:multiplicationCountOfNDSortedGraph},
the LU factorization of the reordered matrix $L_{22}^{\prime}$ incurs a
computational cost of $O(N_2^{\frac{3}{2}}) = O(\frac{1}{h^3})$
by applying the ND ordering in Algorithm \ref{alg:NDOrder} to $L_{22}$.
Figure \ref{fig:Dissect} illustrates the
visual sparse structure of the reordered
matrices $L_{22}^{\prime}$'s resulting from actual computations.
Actually, $L_{22}^{\prime}$'s are
recursively divided into separate sub-blocks,
significantly reducing the number
of fill-ins during Gaussian elimination.

\begin{figure}[htp]
  \centering
  \includegraphics[width=\textwidth]{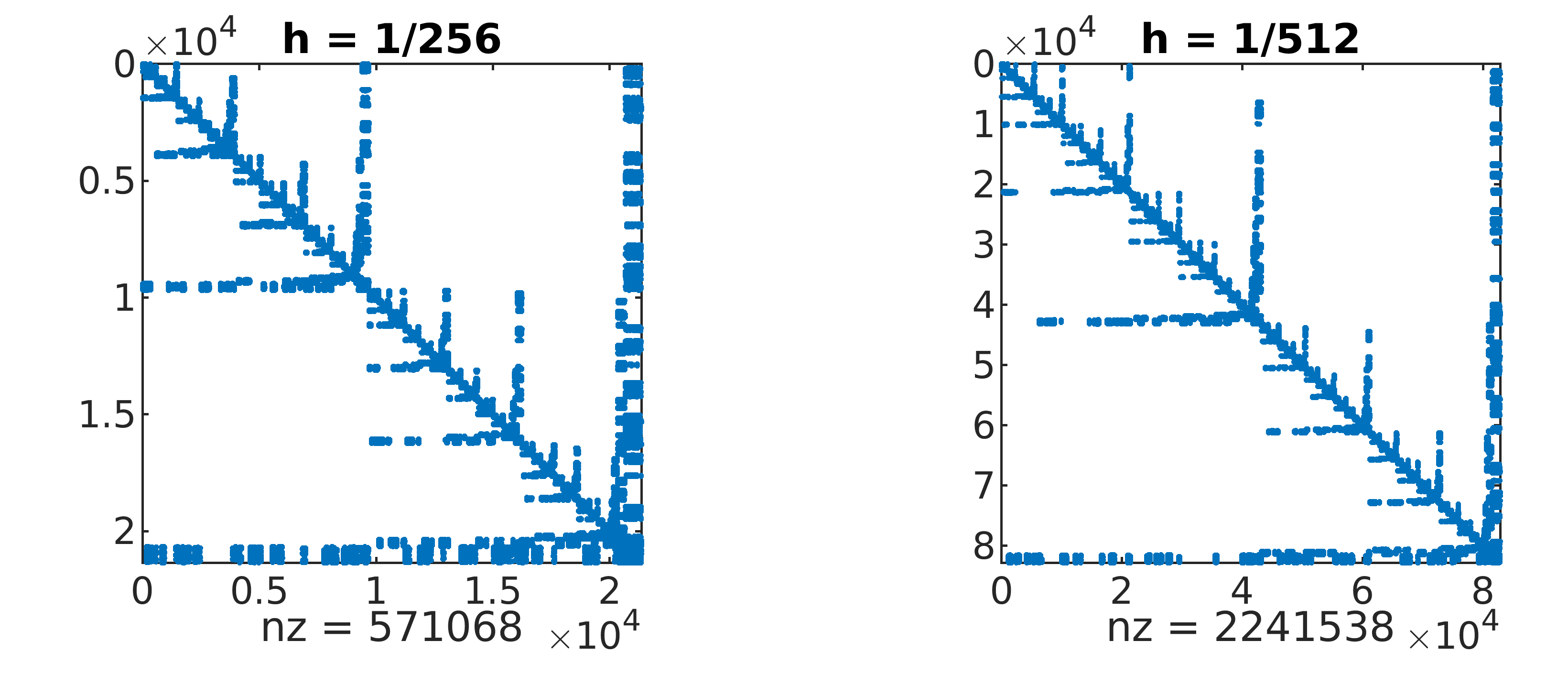}
  \caption{Patterns of nonzero elements in $L_{22}$ after employing
    nested dissection ordering:
    the blue sections indicate the positions of nonzero elements,
    and $nz$ represents the total count of nonzero elements in the matrix.
  }
  \label{fig:Dissect}
\end{figure}

Therefore, the overall complexity of a
single V-cycle (Algorithm \ref{alg:V-cycle}) is
$O(N) = O(\frac{1}{h^3})$, which achieves
the optimal theoretical complexity bound.
Assuming the V-cycle has a
convergence factor $\gamma$ that is independent of $h$,
reducing the solution error from $O(1)$ to $O(h^4)$
requires $O(\log(\frac{1}{h}))$ iterations.
Consequently, the cost of V-cycles is
$O(\frac{1}{h^3}\log(\frac{1}{h}))$.
Moreover, it allows a full multigrid method
(FMG, i.e., Algorithm \ref{alg:revisedFMG})
with optimal complexity $O(\frac{1}{h^3})$.

\begin{algorithm}
	\caption{FMG}
  \label{alg:revisedFMG}
	\begin{algorithmic}[1]
    \REQUIRE
    An integer $M \in \mathbb{Z}^+$ indicates the number of grid levels;
    an integer $m \in \{0,1,\cdots,M\}$ indicates the hierarchy depth;
    the discretization operators of each grid $L^{(m)}$;
    the residual of the $m$th grid $\hat{s}^{(m)}$;
    the number of V-cycles $I_{\text{V-cycle}}$;
    multigrid parameters $\nu_1,\nu_2$.
    \ENSURE Solution for $L^{(m)}\hat{\varphi}^{(m)} = \hat{s}^{(m)}$.
    \IF{$m = M$}
    \STATE Use bottom solver to solve the linear system
    $L^{(M)}\hat{\varphi}^{(M)} = \hat{s}^{(M)}$.
    \RETURN $\hat{\varphi}^{(M)}$.
    \ELSE
    \STATE $\hat{s}^{(m+1)} \leftarrow \mathbf{Restrict}(\hat{s}^{(m)})$.
    \STATE $\hat{\varphi}^{(m+1)}\leftarrow \mathbf{FMG}(\hat{s}^{(m+1)})$.
    \ENDIF
    \STATE $\hat{\varphi}^{(m)}\leftarrow \mathbf{Prolong}(\hat{\varphi}^{(m+1)})$.
    \STATE Perform $I_{\text{V-cycle}}$ V-cycles with initial guess $\hat{\varphi}^{(m)}$.
    \RETURN $\hat{\varphi}^{(m)}$.
	\end{algorithmic}
\end{algorithm}

Given a grid $\Omega^{(m)}$,
denote the linear system as $L^{(m)} \hat{\varphi}^{(m)} = \hat{s}^{(m)}$.
And let $\hat{\varphi}^{(m)}$ and $\hat{\psi}^{(m)}$
denote the exact solution and computed solution of
the linear system, respectively.

\begin{theorem}
  \label{thm:fmgAlgebraicError}
  Suppose the interpolation operator $I_{m+1}^m$ is bounded, i.e.,
  \begin{equation*}
    \exists C > 0, \forall \phi^{(m+1)},
    \|I_{m+1}^m \phi^{(m+1)}\| \leq C \|\phi^{(m+1)}\|,
  \end{equation*}
  and there exists a constant $K \in \mathbb{R}^+$ independent of the grid
  size such that
  \begin{equation*}
    \label{eq:multigridErrorAssumption}
    \|I_{m+1}^m \hat{\varphi}^{(m+1)} - \hat{\varphi}^{(m)} \| \leq K h^p,
  \end{equation*}
  where $h=h^{(m)}$ is the grid size of $\Omega^{(m)}$, and $p$ is
  the order of accuracy of the discrete Laplacian.
  Then a single FMG cycle (Algorithm \ref{alg:revisedFMG}),
  with an appropriate constant $I_{\text{V-cycle}}$,
  reduces the algebraic error from $O(1)$ to $O(h^p)$, i.e.,
  \begin{equation}
    \label{eq:fmgAlgebraicError}
    \|\mathbf{e}^{(m)}\| \leq Kh^p.
  \end{equation}
\end{theorem}

\begin{proof}
  We prove (\ref{eq:fmgAlgebraicError}) by induction.
  On the coarsest grid, FMG is exact and thus (\ref{eq:fmgAlgebraicError})
  holds for the induction basis.
  For the induction hypothesis, we assume that the linear system on $\Omega^{(m+1)}$
  has been solved to the level of discretization error so that
    \begin{equation*}
      \|\mathbf{e}^{(m+1)}\| \leq K(2h)^p.
    \end{equation*}
    Hence, the initial algebraic error on $\Omega^{(m)}$ is
    \begin{equation*}
      \mathbf{e}_0^{(m)} = I_{m+1}^m \hat{\psi}^{(m+1)} - \hat{\varphi}^{(m)},
    \end{equation*}
    which yields
    \begin{align*}
      \|\mathbf{e}_0^{(m)}\|
      &\leq \|I_{m+1}^m \hat{\psi}^{(m+1)} - I_{m+1}^m \hat{\varphi}^{(m+1)}\|
        + \|I_{m+1}^m \hat{\varphi}^{(m+1)} - \hat{\varphi}^{(m)}\|\\
      &\leq C\|\hat{\psi}^{(m+1)} - \hat{\varphi}^{(m+1)}\|
        + \|I_{m+1}^m \hat{\varphi}^{(m+1)} - \hat{\varphi}^{(m)}\|\\
      &\leq CK(2h)^p + Kh^p = (1 + C2^p)Kh^p.
    \end{align*}
    Since $1 + C2^p$ is a constant,
    constant times of V-cycle is enough to reduce
    $\|\mathbf{e}_0^{(m)}\|$ to less than $Kh^p$.
  \end{proof}

  \begin{corollary}
    Under the assumptions of Theorem \ref{thm:fmgAlgebraicError},
    for any $\epsilon > 0$,
    Algorithm \ref{alg:revisedFMG},
    with an appropriate constant $I_{\text{V-cycle}}$,
    can reduce the algebraic error from $O(1)$ to $\epsilon$
    with a complexity of $O(\frac{1}{h^3})$.
  \end{corollary}


\section{Numerical Tests}
\label{sec:numerical-tests}

In this section,
we demonstrate the accuracy and efficiency of our method
by addressing various problems in three-dimensional irregular domains.

\subsection{Geometry Accuracy Tests}
\label{sec:geometry-accuracy-tests}

We first conduct tests on
the accuracy associated with the surface fitting
described in Section~\ref{sec:geometric-characterization}.
We implement the Yin set of the analytic sphere,
which is regarded as the exact boundary here,
and compare it with the surface generated via least squares fitting.
The error norms are defined as
\begin{equation}
  \label{eq:cellFaceAverageNorm}
  \|\mathbf{u}\|_p =
  \begin{cases}
    \left( \frac{1}{N} \sum |\mathbf{u}_i|^p \right)^{\frac{1}{p}}
    &\text{if}\  p = 1,2;\\
    \max |\mathbf{u}_i|
    & \text{if}\ p = \infty,
  \end{cases}
\end{equation}
where $\mathbf{u}$ is a vector with $N$ elements.

Consider a sphere centered at $(0.5, 0.5, 0.5)$ with a radius of $0.2$.
Let $u: \mathbb{R}^3 \rightarrow \mathbb{R}$ be defined by
\begin{equation*}
  u(x,y,z) = 10 x \cdot \sin(y) \cdot e^z.
\end{equation*}
Recalling the descriptions in Section
\ref{sec:geometric-characterization},
we calculate the errors of the
cell-averaged values and face-averaged values
of $u$ associated with $V_f, V_p$ and $S_f, S_p$.
The numerical results presented in Table
\ref{tab:cellAverageErrorsOfSphereWithRadius0.2}
demonstrate that this approximation method achieves $O(h^3)$ accuracy.
The error norms are calculated based on
the error vector of all cut cells
using (\ref{eq:cellFaceAverageNorm}).

\begin{table}[!ht]
  \centering
  \caption{Cell-average and face-average errors of sphere with a radius of $0.2$.}
  \label{tab:cellAverageErrorsOfSphereWithRadius0.2}
  \setlength{\tabcolsep}{3mm}{
    \begin{tabular}{c|c|c|c|c|c|c|c}
      \hline
      \multicolumn{8}{c}{Cell-average errors} \\ \hline
      & $h=\frac{1}{64}$ & rate & $h=\frac{1}{128}$ & rate & $h=\frac{1}{256}$ & rate & $h=\frac{1}{512}$ \\ \hline
      $L^\infty$ &   1.50e-04 &       3.62 &   1.22e-05 &       2.37 &   2.37e-06 &       3.01 &   2.95e-07 \\ \hline
     $L^1$ &   2.50e-05 &       3.41 &   2.36e-06 &       3.12 &   2.73e-07 &       3.02 &   3.36e-08 \\ \hline
     $L^2$ &   4.07e-05 &       3.47 &   3.67e-06 &       3.07 &   4.37e-07 &       3.03 &   5.34e-08 \\ \hline
      \multicolumn{8}{c}{Face-average errors} \\ \hline
      & $h=\frac{1}{64}$ & rate & $h=\frac{1}{128}$ & rate & $h=\frac{1}{256}$ & rate & $h=\frac{1}{512}$ \\ \hline
      $L^\infty$ &   2.22e-06 &       3.42 &   2.08e-07 &       3.26 &   2.17e-08 &      -0.90 &   4.04e-08 \\ \hline
     $L^1$ &   1.31e-07 &       3.37 &   1.26e-08 &       3.23 &   1.35e-09 &       2.75 &   1.99e-10 \\ \hline
     $L^2$ &   2.74e-07 &       3.49 &   2.43e-08 &       3.21 &   2.63e-09 &       2.10 &   6.13e-10 \\ \hline
    \end{tabular}
  }
\end{table}


\subsection{Convergence Tests}

Define the $L^p$ norms as follows:
\begin{equation*}
  \|u\|_p =
  \begin{cases}
    \left( \frac{1}{\|\Omega\|} \sum \|\mathcal{C}_{\mathbf{i}}\|
    \cdot |\langle u\rangle_{\mathbf{i}}|^p \right)^{\frac{1}{p}} & \text{if}\ p = 1,2;\\
    \max |\langle u \rangle_{\mathbf{i}}|  & \text{if}\ p = \infty,
  \end{cases}
\end{equation*}
where the summation and the maximum are taken over the non-empty cells
inside the computational domain.

\subsubsection{Problem1: Sphere Domains}

Consider a problem \cite[Example 5]{gibou2002second}
involving Poisson's equation within a sphere domain,
which centers at $(0.5,0.5,0.5)$ with a radius of 0.3.
The exact solution is given by
\begin{equation*}
  u(x,y,z) = e^{-x^2 - y^2 - z^2}.
\end{equation*}
Dirichlet boundary conditions are applied on all boundary surfaces,
and the unknowns are defined as cell-averaged values.
The solution errors are presented in 
Table \ref{tab:solutionErrorOfSphereWithRadius0.3}.

\begin{table}[!ht]
  \centering
  \caption{Solution errors of sphere with a radius of $0.3$.}
  \label{tab:solutionErrorOfSphereWithRadius0.3}
  \setlength{\tabcolsep}{3mm}{
    \begin{tabular}{c|c|c|c|c|c}
      \hline
      \multicolumn{6}{c}{Solution of the method in \cite{gibou2002second}} \\ \hline
       &       $h = \frac{1}{25}$ &       rate &       $h = \frac{1}{50}$ &       rate &      $h = \frac{1}{100}$ \\ \hline
      $L^\infty$ &   2.27e-04 &       2.12 &   5.20e-05 &       1.99 &   1.31e-05 \\ \hline
     $L^1$ &   6.39e-05 &       1.96 &   1.64e-05 &       2.03 &   4.00e-06 \\ \hline
      \multicolumn{6}{c}{Solution of current method} \\ \hline
       &       $h = \frac{1}{25}$ &       rate &       $h = \frac{1}{50}$ &       rate &      $h = \frac{1}{100}$ \\ \hline
      $L^\infty$ &   6.33e-07 &       4.41 &   2.98e-08 &       4.21 &   1.61e-09 \\ \hline
     $L^1$ &   1.01e-08 &       3.33 &   1.00e-09 &       4.37 &   4.84e-11 \\ \hline
     $L^2$ &   4.15e-08 &       3.50 &   3.67e-09 &       4.35 &   1.80e-10 \\ \hline
    \end{tabular}
  }
\end{table}

\begin{figure}[htp]
  \centering
  \begin{subfigure}[t]{0.48\textwidth}
    \includegraphics[width=0.98\textwidth]{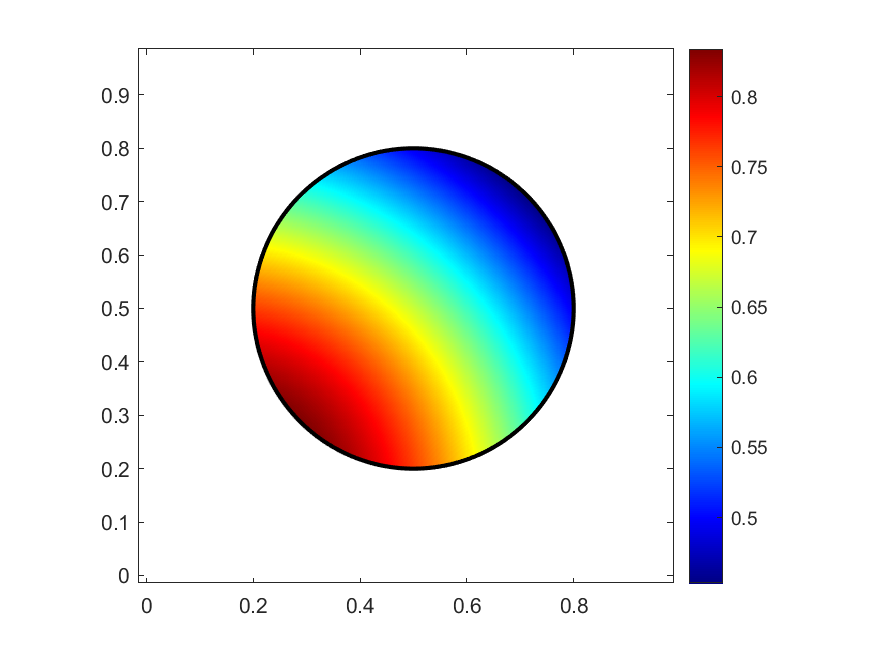}
    \caption{Numerical solution}
    \label{fig:sphereSolutuion}
  \end{subfigure}
  \begin{subfigure}[t]{0.48\textwidth}
    \includegraphics[width=0.98\textwidth]{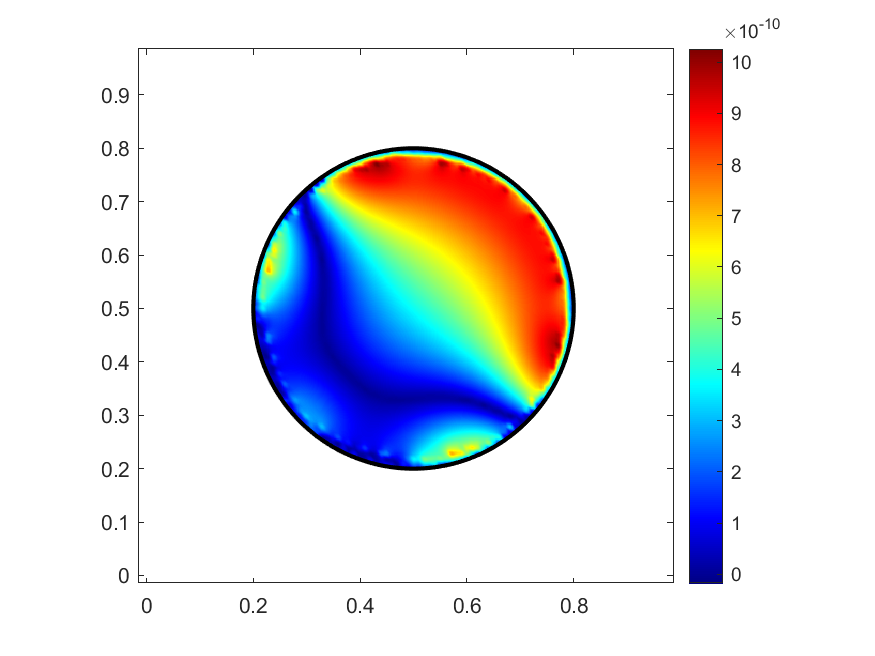}
    \caption{Absolute values of errors}
    \label{fig:sphereError}
  \end{subfigure}
  \caption{Solution and solution errors for sphere 
    with $R = 0.3$, $z = 0.5$, $h = \frac{1}{100}$.}
\end{figure}

\subsubsection{Problem2: Torus Domains}

Consider solving Poisson's equation in
the irregular problem domain $\Omega = B \backslash \Omega_1$,
where $B$ is the unit cube $[0,1]^3$,
and $\Omega_1$ is a torus centered at $(0.5, 0.5, 0.5)$ with a major radius $R = 0.2$
and a minor radius $r = 0.1$.
Unknowns are defined as face-averaged values.
Dirichlet boundary conditions are imposed on the regular boundary surfaces,
and Neumann boundary conditions are imposed on the irregular boundary surfaces.
All the boundary condition values are derived
from the exact solution:
\begin{equation*}
  u(x,y,z) = \cos(\pi x) \cos(\pi y) \sin(\pi z).
\end{equation*}
The truncation errors and solution errors are listed in
Table \ref{tab:errorOfTorusWithR0.2r0.1}.

\begin{table}[!ht]
  \centering
  \caption{Truncation errors and solution errors of torus with $R = 0.2, r = 0.1$.}
  \label{tab:errorOfTorusWithR0.2r0.1}
  \setlength{\tabcolsep}{3mm}{
    \begin{tabular}{c|c|c|c|c|c|c|c}
      \hline
      \multicolumn{8}{c}{Truncation errors} \\ \hline
      & $h=\frac{1}{64}$ & rate & $h=\frac{1}{128}$ & rate & $h=\frac{1}{256}$ & rate & $h=\frac{1}{512}$ \\ \hline
    $L^\infty$ &   9.65e-04 &       1.80 &   2.77e-04 &       2.75 &   4.10e-05 &       2.57 &   6.90e-06 \\ \hline
     $L^1$ &   3.10e-06 &       4.04 &   1.88e-07 &       3.98 &   1.19e-08 &       3.99 &   7.51e-10 \\ \hline
     $L^2$ &   2.44e-05 &       3.34 &   2.41e-06 &       3.51 &   2.12e-07 &       3.51 &   1.86e-08 \\ \hline
      \multicolumn{8}{c}{Solution errors} \\ \hline
      & $h=\frac{1}{64}$ & rate & $h=\frac{1}{128}$ & rate & $h=\frac{1}{256}$ & rate & $h=\frac{1}{512}$ \\ \hline
      $L^\infty$ &   1.97e-07 &       3.75 &   1.47e-08 &       3.75 &   1.09e-09 &       4.14 &   6.22e-11 \\ \hline
     $L^1$ &   7.95e-09 &       3.99 &   5.02e-10 &       3.89 &   3.39e-11 &       3.98 &   2.15e-12 \\ \hline
     $L^2$ &   1.75e-08 &       3.91 &   1.17e-09 &       3.87 &   7.98e-11 &       3.98 &   5.04e-12 \\ \hline
    \end{tabular}
  }
\end{table}

\begin{figure}[htp]
  \centering
  \begin{subfigure}[t]{0.48\textwidth}
    \includegraphics[width=0.88\textwidth]{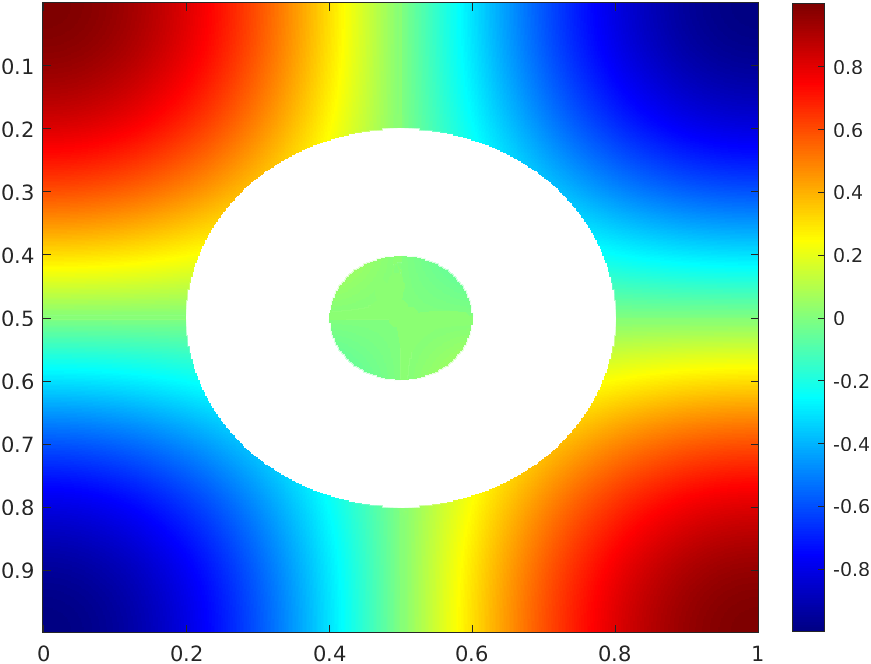}
    \caption{Numerical solution}
    \label{fig:torusSolutuion}
  \end{subfigure}
  \begin{subfigure}[t]{0.48\textwidth}
    \includegraphics[width=0.9\textwidth]{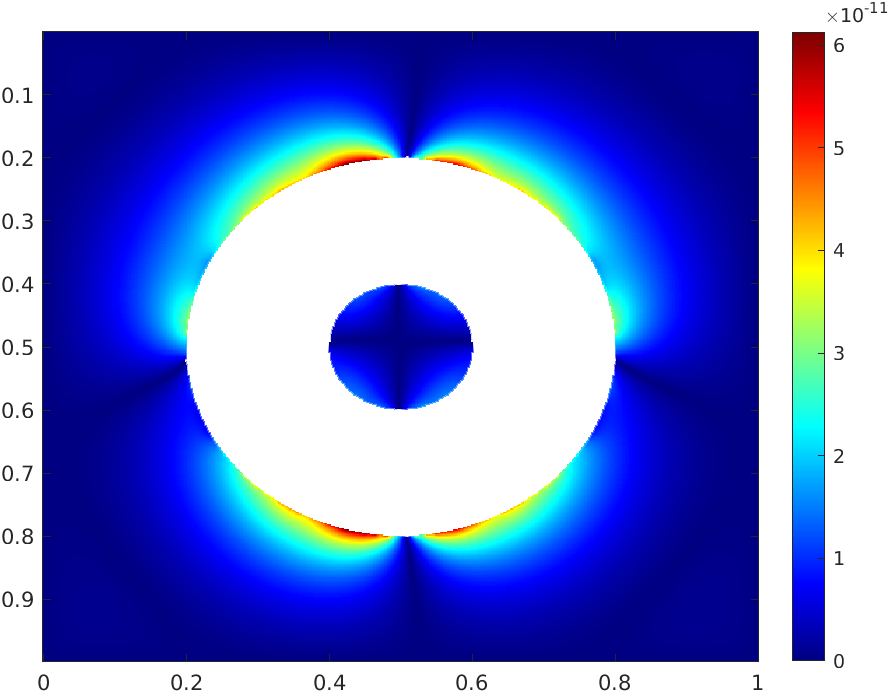}
    \caption{Absolute values of solution errors}
    \label{fig:torusError}
  \end{subfigure}
  \caption{Solution and solution errors for torus 
    with $R = 0.2, r = 0.1$, $z = 0.5$, $h = \frac{1}{512}$.}
\end{figure}

\subsection{Efficiency}
\label{sec:efficiency}
We evaluate the reduction in relative residuals
and the time consumption of
Algorithm \ref{alg:multigrid}.
Figure \ref{fig:reductionOfResidual}
illustrates the reduction of relative residuals
during the solution of Problem 2.
Table \ref{tab:solvingTimeTests}
presents the time consumption of each part of the solution procedure.
The results demonstrate
that the time complexity for both the second and third parts
grows almost cubically.
In summary, the proposed multigrid algorithm efficiently
solves Poisson's equations in complex geometries.

\begin{figure}[h]
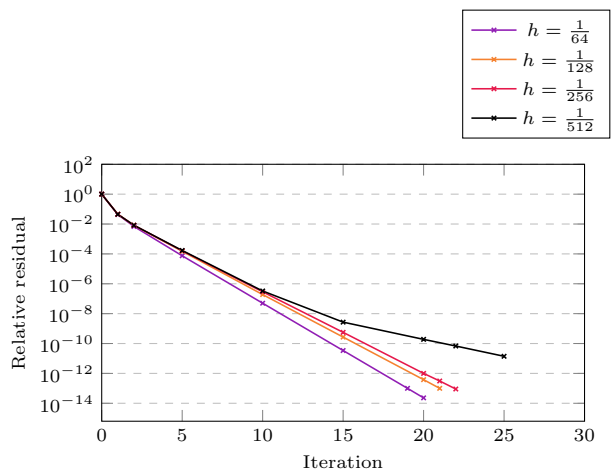

		\centering
		\resizebox{0.8\columnwidth}{!}{
		}
      \includestandalone{{\TIKZDIR}ReductionOfRelativeResidual}
      \caption{Reduction of the relative residual
        ($\frac{\|\hat{s}^{(0)}\|}{\|\hat{r}\|}$ in Algorithm \ref{alg:multigrid}.)
        in Problem 2. The initial guess is the zero function.
        The multigrid parameters are $\omega =0.5$, $\nu_1 = \nu_2 =3$.
      }
		\label{fig:reductionOfResidual}
	\end{figure}

\begin{table}[!ht]
  \centering
  \caption{Time consumption of each stage in the solution procedure.
    The first part "Setup of bottom solver"
    refers to the LU factorization
    of $L^{(M)}$.
    The second part "Setup of LU-correction"
    involves the
    LU factorization of $L^{(m)}_{22}, m=0,\cdots,M-1$.
    After these pre-computations,
    the third part
    "Multigrid solution" follows Algorithm \ref{alg:multigrid}.
    All the tests are run on an AMD Ryzen R9-7950X at 4.5GHz computer
    using single thread,
    and the ND ordering algorithm and LU factorization are
    implemented by Metis
    \cite{karypis1998fast} and PETSc \cite{petsc-user-ref,petsc-web-page}.
  }
  \label{tab:solvingTimeTests}
  \setlength{\tabcolsep}{2.5mm}{
    \begin{tabular}{c|c|c|c|c|c|c|c}
      \hline
      \multicolumn{8}{c}{Solving time for the unit cube with an excluded sphere with $r = 0.3$} \\ \hline
      & $h=\frac{1}{64}$ & rate & $h=\frac{1}{128}$ & rate & $h=\frac{1}{256}$ & rate & $h=\frac{1}{512}$ \\ \hline
        Setup of bottom solver& 8.25  & ~ & 8.10  & ~ & 7.96  & ~ & 7.90  \\ \hline
        Setup of LU-correction& 0.70  & 3.09  & 5.94  & 3.37  & 61.57  & 2.99  & 489.50  \\ \hline
        Multigrid solution& 5.72  & 2.11  & 24.76  & 2.56  & 146.46  & 3.24  & 1385.97 \\ \hline
      \multicolumn{8}{c}{Solving time for the unit cube with an excluded torus with $R = 0.2, r = 0.1$} \\ \hline
      & $h=\frac{1}{64}$ & rate & $h=\frac{1}{128}$ & rate & $h=\frac{1}{256}$ & rate & $h=\frac{1}{512}$ \\ \hline
        Setup of bottom solver& 10.71  & ~ & 10.73  & ~ & 10.52  & ~ & 10.25  \\ \hline
        Setup of LU-correction& 0.43  & 2.90  & 3.21  & 3.26  & 30.85  & 3.15  & 273.04  \\ \hline
        Multigrid solution& 7.32  & 2.82  & 51.60  & 2.88  & 380.84  & 3.01  & 3073.53 \\ \hline
    \end{tabular}
  }
\end{table}


\section{Conclusions}
\label{sec:conclusions}

We have proposed a fourth-order cut-cell method for
solving Poisson's equations in three-dimensional
irregular domains.
Firstly, we use least squares method and technique of Yin space
to characterize arbitrarily complex geometries,
and design an effective merging algorithm for small cells.
Secondly, the FV-PLG algorithm and finite volume method
are applied to derive the high-order discretization of
the Laplacian operator.
Finally, an efficient multigrid algorithm is designed,
which achieves optimal complexity by employing the ND ordering.
The accuracy and efficiency of our method
are demonstrated by numerous numerical tests.

Prospects for future research are as follows.
First, we expect a better boundary geometric representation
which guarantees high-order approximation and
global smoothness by conformal geometry \cite{jin2018conformal}.
Second, we also plan to develop a fourth-order INSE
solver with optimal complexity in three-dimensional irregular domains
based on the GePUP formulation \cite{zhang2016gepup}.



\newpage

\bibliographystyle{siamplain}
\bibliography{references}

\newpage

\end{document}